\documentclass[11pt]{amsart}
\usepackage{amssymb,amscd,amsxtra,calc}
\usepackage{cmmib57}
\usepackage{graphicx}
\usepackage{amssymb,amsmath,amsfonts,mathrsfs,enumerate,xcolor}
\usepackage{amsthm}

\usepackage{soul}
\usepackage[citecolor=Navy]{hyperref}
\usepackage[ruled,vlined]{algorithm2e}

\usepackage{lipsum}
\usepackage{amsfonts}
\usepackage{graphicx}
\usepackage{epstopdf}
\usepackage{algorithmic}
\usepackage{float}
\restylefloat{table}
\usepackage{placeins}
\usepackage{array}
\ifpdf
  \DeclareGraphicsExtensions{.eps,.pdf,.png,.jpg}
\else
  \DeclareGraphicsExtensions{.eps}
\fi

\usepackage{subcaption}

\theoremstyle{plain}
    \newtheorem{thm}{Theorem}[section]

    \newtheorem{lemma}[thm]{Lemma}

    \newtheorem{theorem}[thm]{Theorem}

\theoremstyle{definition}

    \newtheorem{remark}[thm]{Remark}
\theoremstyle{remark}

\newcommand{\authorfootnotes}{\renewcommand\thefootnote{\@fnsymbol\c@footnote}}

\usepackage{caption} 
 \title[A good algorithm for optimization and solving systems of equations]{Backtracking New Q-Newton's method: a good algorithm for optimization and solving systems of equations}
 \author{Tuyen Trung Truong}
   \address{Department of Mathematics, University of Oslo, Blindern 0851 Oslo, Norway}
  \email{tuyentt@math.uio.no}
    \date{\today}
   \subjclass[2010]{}

\begin{document}
\maketitle
{\centering\footnotesize To Professor Eric Bedford, who guided me to Dynamical Systems and the use of computers in research.\par}

\begin{abstract}

In this paper, by combining the algorithm New Q-Newton's method - developed in previous joint work of the author - with Armijo's Backtracking line search, we resolve convergence issues encountered by Newton's method (e.g. convergence to a saddle point or having attracting cycles of more than 1 point) while retaining the quick rate of convergence for Newton's method. We also develop a family of such methods, for general second order methods, some of them having the favour of quasi-Newton's methods. The developed algorithms  are very easy to implement. From a Dynamical Systems' viewpoint, the new iterative method has an interesting feature: while it is deterministic, its dependence on Armijo's Backtracking line search makes its behave like a random process, and thus helps it to have good performance. On the experimental aspect, we compare the performance of our algorithms with well known variations of Newton's method on some systems of equations (both real and complex variables). We also explore some basins of attraction arising from Backtracking New Q-Newton's method, which seem to be quite regular and do not have the fractal structures as observed for the standard Newton's method. Basins of attraction for Backtracking Gradient Descent seem not be that regular. 

\end{abstract}

\section{Introduction} Let $f:\mathbb{R}^m\rightarrow \mathbb{R}$ be a $C^2$ function, with gradient $\nabla f(x)$ and Hessian $\nabla ^2f(x)$.   Newton's method $x_{n+1}=x_n-(\nabla ^2f(x_n))^{-1}\nabla f(x_n)$ (if the Hessian is {\bf invertible}) is a very popular iterative method for solving systems of equations and for optimization in general. It seems that every month there is at least one paper about this subject appears on arXiv. One attractive feature of this method is that if it {\bf converges} then it usually converges very fast, with the rate of convergence being quadratic, which is generally faster than that of Gradient descent (GD) methods.  We recall that if $\{x_n\}\subset \mathbb{R}^m$  converges to $x_{\infty}$, and so that $||x_{n+1}-x_{\infty}||=O(||x_n-x_{\infty}||^{\epsilon})$, then $\epsilon$ is the rate of convergence of the given sequence. If $\epsilon =1$ then we have linear rate of convergence, while if $\epsilon=2$ then we have quadratic rate of convergence.  

However, there is no guarantee that Newton's method will converge, and it is problematic near points where the Hessian is not invertible. Moreover, it cannot avoid saddle points. Recall that a {\bf saddle point} is a point $x^*$ which is a non-degenerate critical point of $f$ (that is $\nabla f(x^*)=0$ and $\nabla ^2f(x^*)$ is invertible) so that the Hessian has at least one {\bf negative eigenvalue}. (Note that this definition allows also local maxima.) Saddle points are problematic in large scale optimization (such as those appearing in Deep Neural Networks, for which the dimensions could easily be millions or billions), see \cite{bray-dean, dauphin-pascanu-gulcehre-cho-ganguli-bengjo}. 

Another serious issue with Newton's method is the existence of attracting cycles of non-critical points, even for simple cost functions. For example, consider the simple polynomial of degree $4$, $f(x)=(x^2-1)(x^2+A)$. For some special values of $A$ (such as $A=(29-\sqrt{720})/11$), Newton's method applied to find roots of $f$ will have an attracting $2$-cycle, none of them is a critical point of $||f||^2$. For polynomials $P(z)$ in 1 complex variable $z$, it is recently shown in \cite{sumi} that random Damped Newton's method $x_{n+1}=x_n-\delta _nP(x_n)/P'(x_n)$, where $\delta _n$ is a random complex number and where the initial point $x_0$ is randomly chosen, will converge to a root of $P(z)$. It is, however, unknown how good random Damped Newton's method is for non-polynomial cost functions and in higher dimensions, and the extensive experiments performed by the author and coauthors in \cite{truong-etal} illustrates that it does not fare well in the general setting. 

Newton's method, and other iterative methods in solving systems of equations and optimization, can be studied in the general setting of discrete Dynamical Systems. There are many modifications of Newton's method. The intended readers can see an overview in \cite{truong-etal}.  Among them, let's mention only the two most relevant to this paper: Levenberg-Marquardt method and Regularized Newton method. 

A well known variant of Newton's method is Levenberg-Marquardt algorithm \cite{levenberg}\cite{marquardt}, very popular for using in the least square fit problem where the cost function $f$ is a sum of squares of real functions: $f(x)=\frac{1}{2}[f_1(x)^2+\ldots +f_N(x)^2]$. Let $F(x)=(f_1(x),\ldots ,f_N(x))$, and $JF(x)$ the Jacobian of $F$. Then the update rule for Levenberg-Marquardt algorithm is: 
\begin{equation}
z_{n+1}=z_n-[JF(z_n)^{T}JF(z_n)+\lambda _nId]^{-1}JF(z_n)^{T}.F(z_n),
\label{EquationLM}\end{equation}
where $A^T$ is  the transpose of a matrix $A$, $\lambda _n>0$ is an appropriate constant. A usual choice for $\lambda _n$ is $\lambda _n=c||F(z_n)||^{\tau}$ for constants $c,\tau  >0$. (One can also interpolate $\lambda _n$ between $||F(z_n)||^{\tau _1}$ and $||JF(z_n)^{T}.F(z_n)||^{\tau _2}$, see \cite{ahookhosh-etal}.)   If in (\ref{EquationLM}) one chooses $\lambda _n=0$, then one obtains the classical Newton's method applied directly to solve the system of equations $F=0$. Hence, Levenberg-Marquardt algorithm can  be viewed as a correction of Newton's method for solving systems of equations, treating the case $JF(z_n)$ not invertible.  

For a general cost function $f$, an analog (or extension) of Levenberg-Marquardt algorithm is the so called Regularized Newton method \cite{ueda-yamashita1}\cite{ueda-yamashita2}\cite{shen-etal}. It replaces the $\nabla ^2f(z_n)$ in Newton's method by $\nabla ^2f(z_n)+[c_1\max \{0,-\lambda _{min}(\nabla ^2f(z_n))\}+\lambda _n]Id$, where $c_1>1$ is a constant, $\lambda _{min}(A)$ is the smallest eigenvalue of a real symmetric matrix $A$, and $\lambda _n>0$ is appropriately chosen.  A usual choice for $\lambda _n$ here is $c_2||\nabla f(z_n)||^{\tau}$ for constants $c_2,\tau >0$.   

Research on these algorithms concentrate only on local convergence (and rate of convergence) near local minima, without much addressing about the global convergence (that is, research on if cluster points of the constructed sequence are all critical points of the cost function, and if the constructed sequence actually converges) or saddle point avoidance. The design of these methods can be a reason for the fact that no result on avoidance of saddle points is available for them (for a more detailed analysis of this issue, see Section 3). In the case of Levenberg-Marquardt method, this can be explained by the fact that the matrix $JF(z_n)^{T}JF(z_n)$ is not the whole of the Hessian of the function $f(x)=\frac{1}{2}||F(x)||^2$, and hence may not be able to control near saddle points of $f(x)$. In the case of Regularized Newton's method, there are  two factors which may contribute to the difficulty when dealing theoretically with this method.  First, the update rule in Regularized Newton's method may not be a $C^1$ map near saddle points where $\lambda _{min}(\nabla ^2f)$ has multiplicity $>1$. Second, the choice of the parameters in this method is not flexible enough. Also, on various examples, these two methods converge only when one adds a line search (e.g. Armijo's Backtracking line search) to them. 

Indeed, the above analysis also applies to many other well known variants of Newton's method in the literature. As far as we know, currently there is no variant which has theoretical guarantees for both convergence and avoidance of saddle points. For a comprehensive test of several variants of Newton's method, on various test problems, the readers can see \cite{truong-etal}. In Section 4 in this paper, we also present some new experiments concerning solving systems of equations. 

For example, these experiments show that the well known Cubic Regularization - and its adaptive version, see \cite{nesterov-polyak}\cite{cartis-etal} - does not fare well, even though it has some good theoretical guarantees. The theoretical guarantees of (adaptive) Cubic Regularization are not stronger than that of Backtracking New Q-Newton's method, even though (adaptive) Cubic Regularization  requires stronger assumptions on the cost functions. For example, in \cite{bianconcini-sciandrone} where line search is integrated into Adaptive Cubic Regularization, whose results more or less represent the current strongest theoretical guarantees of variants of Adaptive Cubic Regularization, the following are proven: i) Assume that for the sequence $\{x_n\}$ constructed  the sequence $||\nabla ^2f(x_n)||$  is uniformly bounded, then we have $\liminf _{n\rightarrow\infty}||\nabla f(x_n)||=0$; ii) If one assumes further that the sequences $f(x_n)$ and  $\nabla f(x_n)$ is uniformly continuous on the sequence $x_n$, then $\lim _{n\rightarrow\infty}\nabla f(x_n)=0$. These conclusions are much weaker than that in our main theorems above. Moreover, unlike many other methods, the performance of installations of Adaptive Cubic Regularization  seems very sensitive to the choice of its  parameters. One probable reason for this is the fact that the algorithm relies on a suboptimal problem which  is quite complicatedly depending on the parameters in order to have good theoretical guarantees. See \cite{ARCGitHub} for a discussion on how difficult is is to actually implement this method, and also for the only publicly available code for Adaptive Cubic Regularization. It is noteworthy to mention that, as observed in \cite{nesterov-polyak}, to precisely solve the optimal subproblem in each step of Cubic Regularization, one actually needs to compute the eigenvalues and eigenvectors of the Hessian matrix, and hence the true complexity of this method is about the same that of our algorithm Backtracking New Q-Newton's method; while as discussed before our method has better theoretical guarantees, better experimental performances and is more flexible. 

This paper shows that a modification (Backtracking New Q-Newton's method) of a very new version of Newton's method, so-called New Q-Newton's method, introduced by the author and collaborators in \cite{truong-etal}, resolves these issues. The method is also straight forward to be implemented from the pseudo-code, and very flexible with respect to its parameters.  For the sake of comprehension,  here we describe the essence of these new algorithms, and present the details of the algorithms in the next section. 

We start with New Q-Newton's method. The main idea is that if we fix $m+1$ distinct numbers $\delta _0,\ldots ,\delta _{m}$, and fix $\tau >0$,  then for every $x$ so that $\nabla f(x)\not= 0$, then there is one $j$ so that $\nabla ^2f(x)+\delta _j ||\nabla f(x)||^{\tau }Id$ is invertible. We choose such a $j$, and change the negative eigenvalues of the matrix to their negative, before taking its inverse in the update rule. If $\delta _0,\ldots ,\delta _m$ are chosen randomly, then  \cite{truong-etal, truongnew} show that New Q-Newton's method avoids saddle points. Moreover, near non-degenerate local minima, New Q-Newton's method has quadratic rate of convergence. However, it is unknown whether New Q-Newton's method has good convergence guarantees. 

In this paper, we observes that the matrix in New Q-Newton's method is semi-positive and hence one can incorporate Armijo's Backtracking line search \cite{armijo} into it. We call the new algorithm Backtracking New Q-Newton's method. It preserves the good properties of New Q-Newton's method, while also resolves the convergence issue.  Moreover, a family of such modifications can be introduced, some of them having the flavours of quasi-Newton's methods, with strong theoretical guarantees. These more general modifications will be presented in the last section of this paper. We conclude this introduction with some main properties of Backtracking New Q-Newton's method.  

We recall that a function $f$ is Morse if all of its critical points are non-degenerate (i.e. if $\nabla f(x_0)=0$, then $\nabla ^2f(x_0)$ is invertible). By transversality results, Morse functions are dense in the set of functions. 

\begin{theorem}
Let $f:\mathbb{R}^m\rightarrow \mathbb{R}$ be a $C^3$ function. Let $\{x_n\}$ be a sequence constructed by the Backtracking New Q-Newton's method, where $0<\tau <1$. 

1) (Descent property) $f(x_{n+1})\leq f(x_n)$ for all n. 

2) If $x_{\infty}$ is a {\bf cluster point} of $\{x_n\}$, then $\nabla f(x_{\infty})=0$. That is, $x_{\infty}$ is a {\bf critical point} of $f$.

3) If $f$ is a Morse function, then either $\lim _{n\rightarrow\infty}||x_n||=\infty$ or $\{x_n\}$ converges to a critical point of $f$. In the latter case, if the initial point $x_0$ is randomly chosen, then the limit point must be a local minimum. 

4) There is a set $\mathcal{A}\subset \mathbb{R}^m$ of Lebesgue measure $0$, so that if $x_0\notin \mathcal{A}$, and $x_n$ converges to $x_{\infty}$, then $x_{\infty}$ cannot be  a {\bf saddle point} of $f$. Hence, if $\nabla ^2f(x_{\infty})$ is invertible, then $x_{\infty}$ must be a local minimum. 

5) If $x_n$ converges to $x_{\infty}$ which is a non-degenerate critical point of $f$, then the rate of convergence is at least {\bf linear}.  

If moreover, $x_{\infty}$ is a local minimum, then the rate of convergence is quadratic.  

6) If $x_{\infty}'$ is a non-degenerate local minimum of $f$, then for initial points $x_0'$ close enough to $x_{\infty}'$, the sequence $\{x_n'\}$  constructed by the Backtracking New Q-Newton's method  will converge to $x_{\infty}'$. 

\label{Theorem1}\end{theorem}

Note that the condition for a function to be  Morse is a generic condition. Hence, for a generic function, and for a randomly chosen initial point $x_0$, our algorithm Backtracking New Q-Newton's method either converges to a local minimum or diverges to infinity.  Another important class of functions, which included real analytic functions, has many realistic applications (such as in Deep Learning). The next main theorem describes convergent guarantees of Backtracking New Q-Newton's method for this class of functions. To this end, we first recall about Lojasiewicz gradient inequality and Lojasiewicz exponent.   

{\bf Lojasiewicz gradient inequality.} A function $f$ satisfies Lojasiewicz gradient inequality at a point $x^*$ if there is a small neighbourhood $U$ of $x^*$, a constant $0<\mu <1$ and a constant $C>0$ so that for all $x,y\in U$ we have
\begin{eqnarray*}
|f(x)-f(y)|^{\mu}\leq C||\nabla f(x)||. 
\end{eqnarray*}

We will need the following quantity in statements of the next results:

{\bf Definition (Lojasiewicz exponent):} Assume that $f:\mathbb{R}^m\rightarrow \mathbb{R}$ has the Lojasiewicz gradient inequality near its critical points. Then at each critical point $x^*$ of $f$, we define

$\mu (x^*):=\inf \{\mu : $ there is an open neighbourhood $U$ of $x^*$ and a constant $C>0$ so that for all $x,y\in U$ we have $|f(x)-f(y)|^\mu \leq C||\nabla f(x)||\}$.  

We say that the gradient $\nabla f$ satisfies Lojasiewicz gradient inequality at a point $x^*$, if  the function $F(x,y)=<\nabla f(x),y>$ $:\mathbb{R}^{2m}\rightarrow \mathbb{R}$ satisfies Lojasiewicz gradient inequality. By Lojasiewicz' theorem, if $f$ is real analytic (and hence $F(x,y)$ is also real analytic), then $f$ and its gradient satisfy Lojasiewicz gradient inequality. Hence, the next theorem can be applied to {\bf quickly} finding roots of systems of (real or complex) analytic equations. Part 3 of its in particularly generalises a result in \cite{truong2021}, which treats the case of finding roots of univariate meromorphic functions (i.e. meromorphic functions in 1 variable). 

\begin{theorem} Assume that $f:\mathbb{R}^m\rightarrow \mathbb{R}$ satisfies the Lojasiewicz gradient inequality. Let $\{x_n\}$ be a sequence constructed by New Q-Newton's Backtracking G. Assume also that $0<\tau \leq 1$. 

1) Assume that $f$ has at most countably many critical points, and $\nabla f$ satisfies the Lojasiewicz gradient inequality. Then either $\lim _{n\rightarrow\infty}||x_n||=\infty$ or $\{x_n\}$ converges to a critical point of $f$. 

2) Assume that for all critical points $x^*$ of $f$, we have $ \mu (x^*)\times (1+\tau ) <1$. For $\tau =1$, assume  moreover that $\nabla f$ also satisfies the Lojasiewicz gradient inequality. Then either $\lim _{n\rightarrow\infty}||x_n||=\infty$ or $\{x_n\}$ converges to a critical point of $f$. 

\label{Theorem2}\end{theorem} 
The result in part 1 is new only when $\mu =1$ (the remaining case is covered by Theorem \ref{Theorem1}). In part 2, we do not require that $f$ has only countably many critical points. In part 2), it is allowed that $\sup _{\nabla f(x^*)=0}\mu (x^*)\times (1+\tau )=1$, provided the supremum is not attained at any critical point $x^*$. On the other hand, since in general $\mu (x^*)\geq 1/2$, it follows that if $\mu (x^*)(1+\tau )<1$ then we should have $\tau <1$. 
 
In \cite{truong-etal}, the case where $f:\mathbb{R}^2\rightarrow \mathbb{R}$ is defined as $f=u^2+v^2$, where $u=$ the real part of $g$ and $v=$ the imaginary part of $g$, $g$ is a univariate holomorphic (more generally, meromorphic) function and $\tau$ can be $\geq 1$, was treated.  A generalisation of this result, which is a bit complicated to state here, will be presented in Section 3.  

{\bf Python code:} One practical usefulness of the algorithms developed in this paper is that they  are very easy to implement, literally as the pseudo code (and even simplified implementations work pretty well). A python code is available here \cite{tuyenGitHub}.

 {\bf Organisation of the paper:} The remaining of this paper is organised as follows. In Section 2, we recall about New Q-Newton's method method, and define the new algorithm Backtracking New Q-Newton's method. In Section 3, we present the proofs of Theorems \ref{Theorem1} and \ref{Theorem2}, as well as some extensions to solving systems of equations (including some new algorithms resulting from a combination between ideas of Backtracking New Q-Newton's method and that of Regularized Newton's method and Levenberg-Marquardt's method).  This is followed by Section 4, where we present some experimental results comparing the long-term global behaviour of our methods with well known ones (such as BFGS, Levenberg-Marquard's algorithm, Regularized Newton's method and (Adaptive) Cubic Regularization of Newton's method) on some polynomial and non-polynomial systems of equations in several variables. In the last section, we provide some further extensions of the algorithms and the results, as well as some conclusions and plan for future work. We display in this section some examples for basins of attraction for Backtracking New Q-Newton's method, which seem quite regular and do not have fractal structures observed in the standard Newton's method.  Basins of attraction for Backtracking Gradient Descent, on the same example, seem not be that regular. This suggest that to actually prove the regularity observed for Backtracking New Q-Newton's method, one needs to make use of the properties of New Q-Newton's method, and that of Backtracking Gradient Descent method alone are not enough.  In the appendix, we describe in detail how Backtracking New Q-Newton's method work in finding roots of a real function in 1-dimension $F:\mathbb{R}\rightarrow \mathbb{R}$. 
  
{\bf Remark:} This paper is developed from the author's 3 preprints: arXiv:2108.10249, arXiv: 2109.11395, and arXiv:2110.07403, and supersede these papers. 

{\bf Acknowledgments.} The author is partially supported by Young Research Talents grant 300814 from Research Council of Norway. The author thanks Thu Hien To and Vu Dinh  for helping with experiments, thanks Jonathan Hauenstein for helping with some relevant questions, and thanks Cinzia Bisi and Stefano Luzzatto for inspiring discussions on the subject. The recent event in Kongsberg,  which happened when part of this work was being done, has been influential to the author.  

\section{Algorithms} In this section we  first recall the algorithm New Q-Newton's method in \cite{truong-etal}, and then define the new algorithm Backtracking New Q-Newton's method and one simplified version of its. 

Let $A:\mathbb{R}^m\rightarrow \mathbb{R}^m$ be an invertible {\bf symmetric} square matrix. In particular, it is diagonalisable.  Let $V^{+}$ be the vector space generated by eigenvectors of positive eigenvalues of $A$, and $V^{-}$ the vector space generated by eigenvectors of negative eigenvalues of $A$. Then $pr_{A,+}$ is the orthogonal projection from $\mathbb{R}^m$ to $V^+$, and  $pr_{A,-}$ is the orthogonal projection from $\mathbb{R}^m$ to $V^-$. As usual, $Id$ means the $m\times m$ identity matrix.

\medskip
{\color{blue}
 \begin{algorithm}[H]
\SetAlgoLined
\KwResult{Find a minimum of $f:\mathbb{R}^m\rightarrow \mathbb{R}$}
Given: $\{\delta_0,\delta_1,\ldots, \delta_{m}\}\subset \mathbb{R}$\ (chosen {\bf randomly}) and $\alpha >0$;\\
Initialization: $x_0\in \mathbb{R}^m$\;
 \For{$k=0,1,2\ldots$}{ 
    $j=0$\\
    \If{$\|\nabla f(x_k)\|\neq 0$}{
   \While{$\det(\nabla^2f(x_k)+\delta_j \|\nabla f(x_k)\|^{1+\alpha}Id)=0$}{$j=j+1$}}

$A_k:=\nabla^2f(x_k)+\delta_j \|\nabla f(x_k)\|^{1+\alpha}Id$\\
$v_k:=A_k^{-1}\nabla f(x_k)=pr_{A_k,+}(v_k)+pr_{A_k,-}(v_k)$\\
$w_k:=pr_{A_k,+}(v_k)-pr_{A_k,-}(v_k)$\\
$x_{k+1}:=x_k-w_k$
   }
  \caption{New Q-Newton's method} \label{table:alg}
\end{algorithm}
}
\medskip
 
 Experimentally, on small scale problems, New Q-Newton's method works quite competitive against the methods mentioned above, see \cite{truong-etal}.  However, while it can avoid saddle points, it does not have a descent property, and an open question in \cite{truong-etal} is whether it has good convergence guarantees. We now define a new algorithm, Backtracking New Q-Newton's method, incorporating Armijo's  Backtracking line search into New Q-Newton's method, which resolves this convergence issue. Backtracking New Q-Newton's method uses  hyperparameters in a more sophisticated manner than that of New Q-Newton's method, and we need some notations.  For a symmetric, square real matrix $A$, we define: 
  
  $sp(A)=$ the maximum among $|\lambda |$'s, where $\lambda  $ runs in the set of eigenvalues of $A$, this is usually called the spectral radius in the Linear Algebra literature;
  
  and 
  
  $minsp(A)=$ the minimum among $|\lambda |$'s, where $\lambda  $ runs in the set of eigenvalues of $A$, this number is non-zero precisely when $A$ is invertible.
  
 One can easily check the following more familiar formulas: $sp(A)=\max _{||e||=1}||Ae||$ and $minsp(A)=\min _{||e||=1}||Ae||$, using for example the fact that $A$ is diagonalisable.  
 
 \medskip
{\color{blue}
 \begin{algorithm}[H]
\SetAlgoLined
\KwResult{Find a minimum of $f:\mathbb{R}^m\rightarrow \mathbb{R}$}
Given: $\{\delta_0,\delta_1,\ldots, \delta_{m}\} \subset \mathbb{R}$\ (chosen {\bf randomly}), $0<\tau $ and $0<\gamma _0<1$;\\
Initialization: $x_0\in \mathbb{R}^m$\;
$\kappa:=\frac{1}{2}\min _{i\not=j}|\delta _i-\delta _j|$;\\
 \For{$k=0,1,2\ldots$}{ 
    $j=0$\\
  \If{$\|\nabla f(x_k)\|\neq 0$}{
   \While{$minsp(\nabla^2f(x_k)+\delta_j \|\nabla f(x_k)\|^{\tau}Id)<\kappa  \|\nabla f(x_k)\|^{\tau}$}{$j=j+1$}}
  
 $A_k:=\nabla^2f(x_k)+\delta_j \|\nabla f(x_k)\|^{\tau}Id$\\
$v_k:=A_k^{-1}\nabla f(x_k)=pr_{A_k,+}(v_k)+pr_{A_k,-}(v_k)$\\
$w_k:=pr_{A_k,+}(v_k)-pr_{A_k,-}(v_k)$\\
$\widehat{w_k}:=w_k/\max \{1,||w_k||\}$\\
$\gamma :=\gamma _0$\\
 \If{$\|\nabla f(x_k)\|\neq 0$}{
   \While{$f(x_k-\gamma \widehat{w_k})-f(x_k)>-\gamma <\widehat{w_k},\nabla f(x_k)>/3$}{$\gamma =\gamma /3$}}

$x_{k+1}:=x_k-\gamma \widehat{w_k}$
   }
  \caption{Backtracking New Q-Newton's method} \label{table:alg0}
\end{algorithm}
}
\medskip

 In the case $f$ has compact sublevels, one can choose $\hat{w _k}=w_k$ in Backtracking New Q-Newton's method (see the discussion after the proof of Theorem \ref{Theorem1} in Section 3). 
 
In the experiments we also explore a simplified version of Backtracking New Q-Newton's method, named Simplified Backtracking New Q-Newton's method, where instead of projecting to the whole space $A_{k,-}$, we project only to a one-dimensional subvector space corresponding to the smallest negative eigenvalue. Experiments show that this simplified version works similarly to Backtracking New Q-Newton's method. One can also consider the whole eigenspace corresponding to the smallest negative eigenvalue, instead of just one eigenvector corresponding with that eigenvalue. 

 \medskip
{\color{blue}
 \begin{algorithm}[H]
\SetAlgoLined
\KwResult{Find a minimum of $f:\mathbb{R}^m\rightarrow \mathbb{R}$}
Given: $\{\delta_0,\delta_1,\ldots, \delta_{m}\} \subset \mathbb{R}$\ (chosen {\bf randomly}), $0<\tau $ and $0<\gamma _0<1$;\\
Initialization: $x_0\in \mathbb{R}^m$\;
$\kappa:=\frac{1}{2}\min _{i\not=j}|\delta _i-\delta _j|$;\\
 \For{$k=0,1,2\ldots$}{ 
    $j=0$\\
  \If{$\|\nabla f(x_k)\|\neq 0$}{
   \While{$minsp(\nabla^2f(x_k)+\delta_j \|\nabla f(x_k)\|^{\tau}Id)<\kappa  \|\nabla f(x_k)\|^{\tau}$}{$j=j+1$}}
  
 $A_k:=\nabla^2f(x_k)+\delta_j \|\nabla f(x_k)\|^{\tau}Id$\\
$v_k:=A_k^{-1}\nabla f(x_k)=pr_{A_k,+}(v_k)+pr_{A_k,-}(v_k)$\\
$minVec:=0$\\
$minVal:=$ the smallest eigenvalue of $A_k$\\
 \If{$minEig<0$}{minVec:= one unit eigenvector of $A_k$ with eigenvalue minVal}

$w_k:=pr_{A_k,+}(v_k)-<v_k,minVec>minVec$\\
$\widehat{w_k}:=w_k/\max \{1,||w_k||\}$\\
$\gamma :=\gamma _0$\\
 \If{$\|\nabla f(x_k)\|\neq 0$}{
   \While{$f(x_k-\gamma \widehat{w_k})-f(x_k)>-\gamma <\widehat{w_k},\nabla f(x_k)>/3$}{$\gamma =\gamma /3$}}

$x_{k+1}:=x_k-\gamma \widehat{w_k}$
   }
  \caption{Simplified Backtracking New Q-Newton's method} \label{table:algSimple}
\end{algorithm}
}
\medskip

\section{Theoretical results}
We present in this section the proofs of Theorems \ref{Theorem1} and \ref{Theorem2}. We also discuss solving a system of equations, and extend a result in \cite{truong-etal} concerning roots of a meromorphic function in 1 complex variable to higher dimensions. 

\subsection{Proofs of Theorems \ref{Theorem1} and \ref{Theorem2}}

We start with auxilliary lemmas.

\begin{lemma} Let $\delta _0,\ldots ,\delta _m$ be distinct real numbers. Define
\begin{eqnarray*}
\kappa :=\frac{1}{2}\min _{i\not=j}|\delta _i-\delta _j|. 
\end{eqnarray*}
Let $A$ be an arbitrary symmetric $m\times m$ matrix, and $\epsilon $ an arbitrary real number. Then there is $j\in \{0,\ldots ,m\}$ so that $minsp(A+\delta _j\epsilon Id)\geq |\kappa \epsilon|$.

\label{Lemma0}\end{lemma}
 \begin{proof}
 Let $\lambda _1,\ldots ,\lambda _m$ be the eigenvalues of $A$. By definition
 \begin{eqnarray*}
 minsp(A+\delta _j\epsilon Id)=\min _{i=1,\ldots ,m}|\lambda _i+\delta _j\epsilon |. 
 \end{eqnarray*}
 The RHS can be interpreted as the minimum of the distances between the points $-\lambda _i$ and $\delta _j\epsilon$. The set $\{\delta _0\epsilon, \delta _1\epsilon, \ldots ,\delta _m\epsilon \}$ consists of $m+1$ distinct points on the real line, the distance between any 2 of them is $\geq 2|\kappa \epsilon|$.  Hence, by the pigeon hole principle, there is at least one $j$ so that the distance from $-\lambda _i$ to $\delta _j\epsilon $ is $\geq |\kappa \epsilon|$ for all $i=1,\ldots ,m$, which is what needed.  
 \end{proof}

The next lemma considers quantities such as $A_k,v_k,w_k,\ldots $ which appear in Algorithm \ref{table:alg0} for Backtracking New Q-Newton's method.  We introduce hence global notations for them for the ease of exposition. Fix a sequence $\delta _0,\delta _1,\ldots ,\delta _{m}$ of real numbers. Define $\kappa :=\min _{i\not= j}|\delta _i-\delta _j|$. Fix also a number $\alpha >0$, and a $C^2$ function $f:\mathbb{R}^m\rightarrow \mathbb{R}$. For $x\in \mathbb{R}^m$ so that $\nabla f(x)\not= 0$, we define: 
 
 $\delta (x):=\delta _j$, where $j$ is the smallest index for which $minsp (\nabla ^2f(x)+\delta _j||\nabla f(x)||^{1+\tau})\geq \kappa ||\nabla f(x)||^{(\tau )}$; (see Lemma \ref{Lemma0}). 
 
  $A(x):=\nabla ^2f(x)+\delta (x)||\nabla f(x)||^{\tau}$; 
  
  $v(x):=A(x)^{-1}.\nabla f(x)$;
  
  $w(x):=pr_{A(x),+}v(x)-pr_{A(x),-}v(x)$; 
  
  and
  
  $\widehat{w(x)}:=w(x)/\max \{1,||w(x)||\}$. 
  
  \begin{lemma} Let $f$ be a $C^2$ function and $x\in \mathbb{R}^m$ for which $\nabla f(x)\not= 0$. We have:

1) $||\nabla f(x)||/sp(A(x)) \leq ||w(x)||\leq ||\nabla f(x)||/minsp(A(x))$.

2) (Descent direction)  

\begin{eqnarray*}
&&||\nabla f(x)||^2/sp(A(x))\leq <w(x),\nabla f(x)>\leq ||\nabla f(x)||^2/minsp(A(x)), \\
&& minsp(A(x))||w(x)||^2\leq <w(x),\nabla f(x)>\leq sp(A(x))||w(x)||^2.
\end{eqnarray*}  

3) In particular,  Armijo's condition 
\begin{eqnarray*}
f(x-\gamma w(x))-f(x)\leq -\gamma <w(x),\nabla f(x)>/3, 
\end{eqnarray*}
is satisfied for all $\gamma >0$ small enough.

\label{Lemma1}\end{lemma}
\begin{proof}
We denote by $e_1,\ldots ,e_m$ an orthonormal basis of eigenvectors of $A(x)$, and let $\lambda _i$ be the corresponding eigenvalue of $e_i$. If we write: 
$$\nabla f(x)=\sum _{i=1}^ma_ie_i=\sum _{\lambda _i>0}a_ie_i+\sum _{\lambda _i<0}a_ie_i,$$
then by definition
\begin{eqnarray*}
v(x)&=&\sum _{\lambda _i>0}a_ie_i/\lambda _i+\sum _{\lambda _i<0}a_ie_i/\lambda _i=\sum _{\lambda _i>0}a_ie_i/|\lambda _i|-\sum _{\lambda _i<0}a_ie_i/|\lambda _i|,\\
w(x)&=&\sum _{\lambda _i>0}a_ie_i/|\lambda _i|+\sum _{\lambda _i<0}a_ie_i/|\lambda _i|=\sum _{i=1}^ma_ie_i/|\lambda _i|. 
\end{eqnarray*}
 
Hence, by calculation 
\begin{eqnarray*}
<w(x),\nabla f(x)>&=&\sum _{i=1}^ma_i^2/|\lambda _i|,\\
||\nabla f(x)||^2&=&\sum _{i=1}^ma_i^2,\\
||w(x)||^2&=&\sum _{i=1}^ma_i^2/|\lambda _i|^2.\\
\end{eqnarray*}

From this, we immediately obtain 1) and  2).

3)  By Taylor's expansion, for $\gamma $ small enough: 
\begin{eqnarray*}
f(x-\gamma w(x))-f(x)&=&-\gamma <w(x),\nabla f(x)> +\frac{\gamma ^2}{2}<\nabla ^2f(x)w(x),w(x)>+o(||\gamma w(x)||^2)\\
&=&-\gamma <w(x),\nabla f(x)>+\gamma ^2O(<w(x),\nabla f(x)>).
\end{eqnarray*}
The last equality follows from 2), where we know that $||w(x)||^2\sim <w(x),\nabla f(x)>$.  Hence, when $\gamma$ is small enough we obtain 3). 
\end{proof}
  
Fix $0<\gamma _0<1$. The above lemma shows that, for $x$ with $\nabla f(x)\not= 0$,  the following quantity (Armijo's step-size) is a well-defined positive number: 
  
  $\gamma (x):=$ the largest number $\gamma $ among the sequence $\{\gamma _0,\gamma _0/3, \gamma _0/3^2,\ldots ,\gamma _0/3^j,\ldots \}$ for which  
  \begin{eqnarray*}
  f(x-\gamma w(x))-f(x)\leq -\gamma <w(x),\nabla f(x)>/3. 
  \end{eqnarray*}
    
  The next lemma explores properties of this quantity.

  \begin{lemma} Let $\mathcal{B}\subset \mathbb{R}^m$ be a compact set so that $\epsilon =\inf_{x\in \mathcal{B}}||\nabla f(x)||>0$. Then

  1) There is $\tau >0$ for which: 
\begin{eqnarray*}
\sup _{x\in \mathcal{B}}sp(A(x))&\leq& \tau ,\\
\inf _{x\in \mathcal{B}}minsp(A(x))&\geq &1/\tau . 
\end{eqnarray*}

2) $\inf _{x\in \mathcal{B}}\gamma (x)> 0$.  

    \label{Lemma2}\end{lemma}
\begin{proof}

1) Since both $\nabla ^2f(x)$ and $\nabla f(x)$ have bounded norms on the compact set $\mathcal{B}$, and $\delta _j$ belongs to a finite set, it follows easily that $\sup _{x\in \mathcal{B}}sp(A(x))<\infty$. (For example, this is seen by using that for a real symmetric matrix $A$, we have $sp(A)=\max _{||v||=1}|<Av,v>| $.)

By Lemma \ref{Lemma0}, $\inf _{x\in \mathcal{B}}minsp(A(x))\geq \kappa \epsilon ^{1+m}$. 

2) From 1) and Lemma \ref{Lemma1},  there is $\tau >0$ so that for all $x\in \mathcal{B}$ we have
\begin{eqnarray*}
&&||\nabla f(x)||^2/\tau \leq <w(x),\nabla f(x)>\leq \tau ||\nabla f(x)||^2, \\
&& ||w(x)||^2/\tau \leq <w(x),\nabla f(x)>\leq \tau ||w(x)||^2.
\end{eqnarray*}  

From this, using Taylor's expansion as in the proof of part 3) of Lemma \ref{Lemma1} and the assumption that $f$ is $C^2$, it is easy to obtain the assertion that $\inf _{x\in \mathcal{B}}\gamma (x)> 0$. 

\end{proof}  

We recall that there is so-called real projective space $\mathbb{P}^m$, which is a compact metric space and which contains $\mathbb{R}^m$ as a {\bf topological} subspace. If $x,y\in \mathbb{R}^m$ and $d(.,.)$ is the metric on   $\mathbb{P}^m$, then 
\begin{eqnarray*}
d(x,y):=\arccos (\frac{1+<x,y>}{\sqrt{1+||x||^2}\sqrt{1+||y||^2}}). 
\end{eqnarray*} 
It is known that there is a constant $C>0$ so that for $x,y\in \mathbb{R}^m$ we have $d(x,y)\leq C||x-y||$, see e.g. \cite{truong-nguyen2}. Real projective spaces were used in \cite{truong-nguyen1, truong-nguyen2} to establish good convergence guarantees for Backtracking GD and modifications.

 Now we prove Theorem \ref{Theorem1}. 
 \begin{proof}[Proof of Theorem \ref{Theorem1}]

1) This is by construction. 

2) It suffices to show that if $\{x_{n_k}\}$ is a bounded subsequence of $\{x_n\}$, then $\lim _{n\rightarrow\infty}\nabla f(x_{n_k})=0$. If it is not the case, we can assume, by taking a further subsequence if needed, that $\inf _{k}||\nabla f(x_{n_k})||>0$. Then, by Lemma \ref{Lemma2}, upto some positive constants, $<\nabla f(x_{n_k}),w(x_{n_k})>$ $\sim$ $||\nabla f(x_{n_k})||^2$ $\sim$ $||w(x_{n_k})||^2$ for all $k$. Since $||\nabla f(x_{n_k})||$ is bounded, it follows that $w_{n_k}\sim \widehat{w_{n_k}}$.   

It follows from the fact that $\{f(x_n)\}$ is a decreasing sequence, that $\lim _{k\rightarrow\infty}(f(x_{n_k})-f(x_{n_{k}+1}))=0$. It follows from Armijo's condition that $\lim _{k\rightarrow\infty}\gamma _{n_k}<\nabla f(x_{n_k}),\widehat{w(x_{n_k})}>=0$. By Lemma \ref{Lemma2} we have $\inf _{k}\gamma _{n_k}>0$, and hence $\lim _{k\rightarrow\infty}<\nabla f(x_{n_k}),\widehat{w(x_{n_k})}>=0$.  From the equivalence in the previous paragraph about $w_{n_k}\sim \widehat{w_{n_k}}$ and $<\nabla f(x_{n_k}),\widehat{w(x_{n_k})}>\sim ||w_{n_k}||^2\sim ||\nabla f(x_{n_k})||^2$, it follows that $\lim _{k\rightarrow\infty}\nabla f(x_{n_k})=0$, a contradiction.

3) We show first that $\lim _{n\rightarrow\infty}d(x_n,x_{n+1})=0$, where $d(.,.)$ is the projective metric defined above.  It suffices to show that each subsequence $\{x_{n_k}\}$ has another subsequence for which the needed claim holds. By taking a further subsequence if necessary, we can assume that either $\lim _{k\rightarrow\infty}||x_{n_k}||=\infty$ or $\lim _{k\rightarrow\infty}x_{n_k}=x_{\infty}$ exists. 

The first case:  $\lim _{k\rightarrow\infty}||x_{n_k}||=\infty$. In this case, since $x_{n_{k}+1}=x_{n_k}-\gamma _{n_k}\widehat{w_{n_k}}$, where both $\gamma _{n_k}$ and $\widehat{w_{n_k}}$ are bounded, it follows easily from the definition of the projective metric that $\lim _{k\rightarrow \infty}d(x_{n_k+1},x_{n_k})=0$. 

The second case: $\lim _{k\rightarrow\infty}x_{n_k}=x_{\infty}$ exists. In this case, by part 1) of Theorem \ref{Theorem1}, $x_{\infty}$ is a critical point of $f$.  Since $f$ is Morse by assumption, we have that $\nabla ^2f(x_{\infty})$ is invertible. Then, by the proof of part 4) below, we have that $\lim _{k\rightarrow\infty}\widehat{w_{n_k}}=0$. Then, since $\gamma _{n_k}$ is bounded, one obtain that
\begin{eqnarray*}
\lim _{k\rightarrow\infty}||x_{n_k+1}-x_{n_k}||=\lim _{k\rightarrow\infty}||\gamma _{n_k}\widehat{w_{n_k}}||=0. 
\end{eqnarray*}
Thus by the inequality between the projective metric and the usual Euclidean norm, one obtains $\lim _{k\rightarrow\infty}d(x_{n_k+1},x_{n_k})=0$. 

With this claim proven, one can proceed as in \cite{truong-nguyen2}, using part 2) above and the fact that the set of critical points of a Morse function is at most countable, to obtain a bifurcation: either $\lim _{n\rightarrow\infty}x_n$ converges to a critical point of $f$, or $\lim _{n\rightarrow\infty}||x_n||=\infty$.  

If the initial point $x_0$ is randomly chosen, then we know  from 4) below that the limit point $x_{\infty}$ cannot be a saddle point. Since a critical point of a Morse function is either a saddle point or a local minimum, we conclude that $x_{\infty}$ must be a local minimum of $f$. 

4)  Assume that $x_n$ converges to $x_{\infty}$ and $x_{\infty}$ is a saddle point. Then $\nabla ^2f(x_{\infty})$ is invertible while $\nabla f(x_{\infty})=0$. It follows that there is a small open neighbourhood $U$ of $x_{\infty}$ so that for all $x\in U$ we have $\delta (x)=\delta _0$. By shrinking $U$ if necessary, we also have that $||w(x)||\leq \epsilon$ for all $x\in U$, and hence $\widehat{w(x)}=w(x)$, for all $x\in U$. Here $\epsilon >0$ is a small positive number, to be determined later. 

By Taylor's expansion, and using $A(x)=\nabla ^2f(x)+\delta _0||\nabla f(x)||^{\tau}Id=\nabla ^2f(x)+o(1)$, we have 
\begin{eqnarray*}
f(x-w(x))-f(x)&=&-<w(x),\nabla f(x)>+\frac{1}{2}<\nabla ^2f(x)w(x),w(x)>+o(||w(x)||^2)\\
&=&-<w(x),\nabla f(x)>+\frac{1}{2}<A(x)w(x),w(x)>+o(||w(x)||^2).
\end{eqnarray*}

We denote by $e_1,\ldots ,e_m$ an orthonormal basis of eigenvectors of $A(x)$, and let $\lambda _i$ be the corresponding eigenvalue of $e_i$. If we write: 
$$\nabla f(x)=\sum _{i=1}^ma_ie_i=\sum _{\lambda _i>0}a_ie_i+\sum _{\lambda _i<0}a_ie_i,$$
then by definition
\begin{eqnarray*}
v(x)&=&\sum _{\lambda _i>0}a_ie_i/\lambda _i+\sum _{\lambda _i<0}a_ie_i/\lambda _i=\sum _{\lambda _i>0}a_ie_i/|\lambda _i|-\sum _{\lambda _i<0}a_ie_i/|\lambda _i|,\\
w(x)&=&\sum _{\lambda _i>0}a_ie_i/|\lambda _i|+\sum _{\lambda _i<0}a_ie_i/|\lambda _i|=\sum _{i=1}^ma_ie_i/|\lambda _i|. 
\end{eqnarray*}

Then $A(x)w(x)=\sum _{\lambda _i>0}a_ie_i-\sum _{\lambda _i<0}a_ie_i$.  Therefore,
\begin{eqnarray*}
\frac{1}{2}<A(x)w(x),w(x)>&=&\frac{1}{2}<\sum _{\lambda _i>0}a_ie_i-\sum _{\lambda _i<0}a_ie_i,\sum _{i=1}^ma_ie_i/|\lambda _i|>\\
&=&\frac{1}{2}(\sum _{\lambda _i>0}a_i^2/|\lambda _i|-\sum _{\lambda _i<0}a_i^2/|\lambda _i|)\\
&\leq&\frac{1}{2}\sum _{i=1}^ma_i^2/|\lambda _i|\\
&=&\frac{1}{2}<w(x),\nabla f(x)>. 
\end{eqnarray*}

Combining all the above, we obtain: 
\begin{eqnarray*}
f(x-w(x))-f(x)\leq -\frac{1}{2}<w(x),\nabla f(x)>+o(||w(x)||^2),
\end{eqnarray*}
for all $x\in U$. By Lemma \ref{Lemma1}, for $x\in U$ we have $<w(x),\nabla f(x)>\sim ||w(x)||^2$. Hence, $\gamma (x)=\gamma _0$ for all $x\in U$. 

This means that $x_{n+1}=x_n-\gamma _0w_n$ for $n$ large enough, that is, Backtracking New Q-Newton's method becomes New Q-Newton's method. Then, for the case $\tau >1$, the results in \cite{truong-etal, truongnew} finish the proof that $x_{\infty}$ cannot be a saddle point. The case $0<\tau \leq 1$ is similarly proven: the same integral representation of $pr_{A(x),\pm}$ shows again that the associated dynamical system is $C^1$, and hence arguments in \cite{truong-etal, truongnew}  apply. 

5) and 6): follow easily from the corresponding parts in \cite{truong-etal} and the above parts. 
\end{proof}
 
 \begin{proof}[Proof of Theorem \ref{Theorem2}]
1) As in part 2) of Theorem \ref{Theorem1},  it suffices to show that  for a subsequence $x_{n_k}$ converging to a critical point $x^*$: $\lim_{k\rightarrow\infty}||w_{n_k}||=0$. We look at a new function $F:\mathbb{R}^{2m}\rightarrow \mathbb{R}$ given by $F(x,y)=<\nabla f(x),y>$. We have $F(x=x^*,y=0)=0$. Moreover, $\nabla F=(\nabla ^2f(x).y,\nabla f(x))$.  Hence, by assumption, there exists $C>0$ and $\mu <1$, and an open neighborhood $U$ of $x^*$ as well as $\epsilon >0$ so that for all $x\in U$ and all $||y||=1$:
\begin{eqnarray*}
|<\nabla f(x),y>|^{\mu }\leq C(||\nabla ^2f(x).y||+||\nabla f(x)||). 
\end{eqnarray*}

We let $\{e_{j,k}\}_{j=1,\ldots ,m}$ be an orthonormal basis of $A_k:=A(x_k)$, with the corresponding eigenvalues $\{\lambda _{j,k}\}_{j=1,\ldots ,m}$. Then, by replacing $C$ with a bigger constant, and substituting $x$ by $x_{n_k}$ and $y$ by $e_{j,k}$ we obtain also the inequality
\begin{eqnarray*}
|<\nabla f(x_{n_k}),e_{j,k}>|^{\mu }\leq C(||\nabla ^2f(x_{n_k}).e_{j,k}+\delta ||\nabla f(x_{n_k})||e_{j,k}||+||\nabla f(x_{n_k})||). 
\end{eqnarray*}
If we choose $\delta$ appropriately from the finite set $\delta _0,\ldots ,\delta _m$, we obtain $A(x_{n_k})$.  

Since $\tau \leq 1$ and $||e_{j,k}||=1$, it follows that $||A(x).e_{j,k}||\geq minsp(A(x))\geq \kappa ||\nabla f(x)||$. Therefore, by replacing $C$ by a bigger constant, we obtain that $|<\nabla f(x_{n_k}),e_{j,k}>|^{\mu }\leq C||A(x).e_{j,k}||$ for every $j$. Therefore, we obtain
\begin{eqnarray*}
||w_{n_k}||\leq C||\nabla f(x_{n_k})||^{1-\mu}\rightarrow 0, 
\end{eqnarray*}
as wanted. 

2) Similar to \cite{truong-etal}, we need to show that if $x^*$ is a critical point of $f$, there exists a constant $C>0$ and $\theta <1$ so that if a subsequence $\{x_{n_k}\}$ converges to $x^*$, then:
\begin{eqnarray*}
\lim _{k\rightarrow \infty}||w_{n_k}||&=&0,\\
<w_{n_k},\nabla f(x_{n_k})>&\geq& C||w_{n_k}||\times |f(x_{n_{k+1}})-f(x_{n_k})|^{\theta}.
\end{eqnarray*}

The equality follows from part 2 of Theorem \ref{Theorem1} (if $0<\tau <1$) and from part 1 above if $\tau =1$. 

Now we prove the inequality. We have
\begin{eqnarray*}
||A_{n_k}.e_{j,n_k}||\geq minsp(A_{n_k})\geq \kappa ||\nabla f(x_{n_k})||^{\tau}. 
\end{eqnarray*}

Then, computed as in \cite{truong2021}, we obtain
\begin{eqnarray*}
<w_{n_k},\nabla f(x_{n_k})>\geq C||w_{n_k}||\times ||\nabla f(x_{n_k})||^{1+\tau}. 
\end{eqnarray*}
From this, the claim follows by the definition of $\mu (x^*)$: we can choose $\theta =\mu (1+\tau )<1$ for some choices of $\mu \geq \mu (x^*)$.

 \end{proof}

\subsection{Solving systems of equations}

Now we discuss an application of the above results to solving systems of equations. Assume that we are given a real analytic map $F:\mathbb{R}^m\backslash \mathcal{A}\rightarrow \mathbb{R}^{m'}$, where $\mathcal{A}$ is a closed set of $\mathbb{R}^{m}$ of Lebesgue measure 0 for which: $\lim _{x\rightarrow A}||F(x)||=+\infty$.  We consider the question of finding solutions  to $F(x)=0$. Following the usual trick, we define $f(x)=\frac{1}{2}||F(x)||^2: ~\mathbb{R}^m\rightarrow [0,+\infty ]$. A solution to $F=0$ is a global minimum of $f$, so we can consider the question of finding global minima to $f$ (which still make sense even if the system $F=0$ has no solutions).   Since $\mathcal{A}$ has Lebesgue measure zero and since the sequence constructed by Backtracking New Q-Newton's method has decreasing function values, if one chooses a random point $x_0$ then it will belong to $\mathbb{R}^m\backslash \mathcal{A}$ and the whole sequence $\{x_n\}$ will also belong to $\mathbb{R}^m\backslash \mathcal{A}$. One can apply Theorems \ref{Theorem1} and \ref{Theorem2} to show that a cluster point of $\{x_n\}$ must be a critical point of $f$. Moreover, the constructed sequence $\{x_n\}$ either {\bf converges} to a unique critical point of $f$ or the norms $\{||x_n||\}$ converges to $+\infty$,   provided conditions in those mentioned theorems are satisfied. (For example, if $f$ has at most countably many critical points - which is a generic condition; or if $F$ is a polynomial map and $\tau $ is small enough - depending on the degree of $F$. Indeed, it is shown in \cite{acunto-kurdyka} that the Lojasiewicz exponent of a polynomial map is globally bounded in terms of the map's degree, and hence if we choose $\tau$ small enough then conditions in part 2 of Theorem \ref{Theorem2} are satsified.) Part 4 of Theorem \ref{Theorem1} then allows one to conclude that the limit point of $\{x_n\}$ cannot be a saddle point, hence roughly must be a local minimum of $f$. 

There is need for much more further research, since the above results are still not enough to solve the system $F=0$.  On the one hand, local minima of $f$ may not be global minima. On the other hand, the requirement that $\tau $ is small enough may be difficult to check, since strictly speaking one then must know an upper bound for the global Lojasiewicz exponent, which can be  too small to be useful in numerical calculations, even in the case where $F$ is a polynomial map the bound in \cite{acunto-kurdyka} is extremely small compared with the degree. Note, however, on the experimental side, see the next Section, Backtracking New Q-Newton's method works very well and very flexible on its parameters (for example, the algorithm run wells on various examples even if one chooses $\tau =1$). 

For the case where $g:\mathbb{C}\backslash \mathcal{A}\rightarrow \mathbb{C}$ is a holomorphic function with poles at  $\mathcal{A}$, and one defines $F(x,y)=(Re(g(x+iy)), Im(g(x+iy))):\mathbb{R}^2\backslash \mathcal{A}\rightarrow \mathbb{R}^2$ consisting of the real and imaginary parts of $g(z)$ (where $z=x+iy$), then \cite{truong-etal} shows that the above   obstacles can be resolved provided that $g$ satisfies the following generic condition: For every $z_0\in \mathbb{C}$, if  $g(z_0)g"(z_0)=0$ then $g'(z_0)\not= 0$.  Indeed, under this generic condition, it is proven in \cite{truong-etal} that if one applies Backtracking New Q-Newton's method (no matter what value of $\tau$ is) for the function $f$, from a random initial point $z_0\in \mathbb{R}^2=\mathbb{C}$, then for the constructed sequence $\{z_n\}$ one has a dichotomy: either $\{z_n\}$ converges to a zero $z*$ of $g(z)$, or $\lim _{n\rightarrow\infty}||z_n||=+\infty$.  The main ingredient in \cite{truong-etal} is to make use of Cauchy-Riemann's equations to relate critical points of $f$ which are not zeros of $F$ to saddle points of $f$. 

The following result generalises part of the mentioned result in \cite{truong-etal}. 
\begin{theorem} Assume that $f:\mathbb{R}^m\rightarrow \mathbb{R}$ satisfies the Lojasiewicz gradient inequality. Let $\{x_n\}$ be a sequence constructed by New Q-Newton's Backtracking G. Let $\{e_{k}\}$ be an orthonormal basis for $A_k:=A(x_k)$. Assume also that $0<\tau \leq 1$, and that whenever a subsequence $x_{k_n}$ converges,  we have: 
\begin{eqnarray*}
\liminf _{n\rightarrow\infty}\frac{\min _{i\in \Lambda _{k_n}} ||A_{k_n}.e_{i,k}||}{\max _{i\in \Lambda _{k_n}} ||A_{k_n}.e_{i,k_n}||}>0, 
\end{eqnarray*}
where $\Lambda _{k_n}=\{i:~<\nabla f(x_{k_n}),e_{i,k_n}>\not= 0\}$. (Moreover, if $\tau =1$, then assume that $\nabla f$  also satisfies the Lojasiewicz gradient inequality.) 

Then either $\lim _{n\rightarrow\infty}||x_n||=\infty$ or $\{x_n\}$ converges to a critical point of $f$. . 

\label{Theorem3}\end{theorem}
\begin{proof}
The proof is similar to that in \cite{truong-etal}. 
\end{proof}

One can also combine the ideas of Backtracking New Q-Newton's method with that of Regularized Newton's method an Levenberg-Marquardt's method to obtain methods with better theoretical guarantees and performances, specifically for  systems of equations.  Below we present two such algorithms, together with some theoretical results. Readers can see the next section on how these methods work on solving systems of equations, and see the last section for further generalisations. 

Given $F:\mathbb{R}^m\rightarrow \mathbb{R}^{m'}$ a $C^2$ function, we denote by $H(x)=JF(x)$ the Jacobian of $F$, and $f(x)=||F(x)||^2$. Note that $\nabla f(x)=2H(x)^{\intercal}F(x)$.  Hence the last While loop in the  algorithm  terminates after a finite time, see the previous Subsection. 

\medskip
{\color{blue}
 \begin{algorithm}[H]
\SetAlgoLined
\KwResult{Find a zero of $F:\mathbb{R}^m\rightarrow \mathbb{R}^{m'}$}
Given: $0<\delta_0,\delta_1,\ldots ,\delta _m$  and $0<\tau $;\\
Define: $H(x):=JF(x)$ and $f(x)=||F(x)||^2$;\\
Define: $\kappa :=\inf _{i\not= j}|\delta _i-\delta _j|/2$;\\
Initialization: $x_0\in \mathbb{R}^m$\;
 \For{$k=0,1,2\ldots$}{ 
    $j=0$\\
  \If{$F(x_k)\neq 0$}{
  \If{$minsp(\nabla ^2f(x_k))>||F(x_k)||^{\tau}$}
   {  ~~~~\While{$minsp(\nabla ^2f(x_k)+\delta _j||F(x_k)||)<\kappa ||F(x_k)||$}{j:=j+1}
   ~~~$A_k:=\nabla ^2f(x_k)+\delta _j||F(x_k)||\times Id$\\
  {\bf else}\\
  ~~~~\While{$minsp(\nabla ^2f(x_k)+\delta _j||F(x_k)||^{\tau})<\kappa ||F(x_k)||^{\tau}$}{j:=j+1}
   ~~~$A_k:=\nabla ^2f(x_k)+\delta _j||F(x_k)||^{\tau}\times Id$}
  }
$v_k:=A_k^{-1}H(x_k)^{\intercal}F(x_k)=pr_{A_k,+}v_k+pr_{A_k,-}v_k$\\
 $w_k:=pr_{A_k,+}v_k-pr_{A_k,-}v_k$\\
$\widehat{w_k}:=w_k/\max \{1,||w_k||\}$\\
$\gamma :=1$\\
 \If{$H(x_k)^{\intercal}F(x_k)\neq 0$}{
   \While{$f(x_k-\gamma \widehat{w_k})-f(x_k)>-\gamma <\widehat{w_k},H(x_k)^{\intercal}F(x_k)>$}{$\gamma =\gamma /2$}}

$x_{k+1}:=x_k-\gamma \widehat{w_k}$
   }
  \caption{Backtracking New Q-Newton's method SE} \label{table:alg3}
\end{algorithm}
}
\medskip

\medskip
{\color{blue}
 \begin{algorithm}[H]
\SetAlgoLined
\KwResult{Find a zero of $F:\mathbb{R}^m\rightarrow \mathbb{R}^{m'}$}
Given: $0<\delta_0,\delta_1$  and $0<\tau $;\\
Define: $H(x):=JF(x)$ and $f(x)=||F(x)||^2$;\\
Define: $\kappa :=|\delta _1-\delta _0|/2$;\\
Initialization: $x_0\in \mathbb{R}^m$\;
 \For{$k=0,1,2\ldots$}{ 
    $j=0$\\
  \If{$F(x_k)\neq 0$}{
  \If{$minsp(H(x_k)^{\intercal}H(x_k))>||F(x_k)||^{\tau}$}
   {$A_k:=H(x_k)^{\intercal}H(x_k)+\delta _0||F(x_k)||\times Id$\\
  {\bf else}\\
   $A_k:=H(x_k)^{\intercal}H(x_k)+\delta _1||F(x_k)||^{\tau}\times Id$}
  }
 $w_k:=A_k^{-1}H(x_k)^{\intercal}F(x_k)$\\
$\widehat{w_k}:=w_k/\max \{1,||w_k||\}$\\
$\gamma :=1$\\
 \If{$H(x_k)^{\intercal}F(x_k)\neq 0$}{
   \While{$f(x_k-\gamma \widehat{w_k})-f(x_k)>-\gamma <\widehat{w_k},H(x_k)^{\intercal}F(x_k)>$}{$\gamma =\gamma /2$}}

$x_{k+1}:=x_k-\gamma \widehat{w_k}$
   }
  \caption{Backtracking Levenberg-Marquardt's algorithm}  \label{table:alg4}
\end{algorithm}
}
\medskip

We have the following result. Note that by definition, for the $A_k$ in either algorithm, we have $minsp(A_k)\geq ||F(x_k)||^{\tau }$, and near a non-degenerate root of $F(x)$ we have $A_k=H(x_k)^{\intercal}H(x_k)+O(||F(x_k)||)$. 

\begin{theorem}

Let $F:\mathbb{R}^m\rightarrow \mathbb{R}^{m'}$ be a $C^1$ function. Define $H(x)=JF(x)$ and $f(x)=||F(x)||^2$. Let $x_0\in \mathbb{R}^m$ be an initial point, and $\{x_n\}$ the corresponding constructed sequence from New Q-Newton's method Backtracking SE or Levenberg-Marquardt M.  

0) (Descent property) $f(x_{n+1})\leq f(x_n)$ for all n. 

1) If $x_{\infty}$ is a {\bf cluster point} of $\{x_n\}$, then $H(x_{\infty})^{\intercal}F(x_{\infty})=0$. That is, $x_{\infty}$ is a {\bf critical point} of $f$.

2) If $0<\tau <1$: If $f$ has at most countably many critical points, then either $\lim _{n\rightarrow\infty}||x_n||=\infty$ or $\{x_n\}$ converges to a critical point of $f$. Moreover, if $f$ has compact sublevels, then only the second alternative happens. 

3) If $x_n$ converges to $x_{\infty}$ which is a non-degenerate zero  of $F$ (that is, if $H(x)$ is invertible at $x_{\infty}$), then the rate of convergence is quadratic. 

4) If $0<\tau <1$: (Capture theorem) If $x_{\infty}'$ is an isolated zero of $F$, then for initial points $x_0'$ close enough to $x_{\infty}'$, the sequence $\{x_n'\}$  constructed by New Q-Newton's method Backtracking SE will converge to $x_{\infty}'$. 

\label{Theorem7}\end{theorem}

\begin{proof}

Parts 0, 1 and 2 follow exactly as in  for Backtracking New Q-Newton's method. 

Part 3: this follows because near $x_{\infty}$ then the update rule for either algorithm is different from the usual Newton's method only in $O(||F(x_k)||)$, which will still has quadratic rate of convergence. 

Part 4: this is well known for optimization algorithms having the descent property in part 0.

\end{proof}

Next, we discuss the avoidance of saddle point. Recall that a point $x^*$ is a generalised saddle point of a function $f$ if $\nabla f(x^*)=0$, and moreover $\nabla ^2f(x^*)$ has at least one negative eigenvalue. Note that if $x^*$ is a generalised saddle point of $f(x)=||F(x)||^2$, then $x^*$ cannot be a root of $F(x)=0$. In particular, since $\nabla f(x)/2=H(x)^{\intercal}F(x)$, we have that $H(x^*)$ is singular, and hence $A(x)=H(x)^{\intercal}H(x)+\delta _1||F(x)||^{\tau}$ near $x^*$. 

We consider the Backtracking Levenberg-Marquardt's method first, which is more complicated to deal with. The main reason is that the main term $H(x)^{\intercal}H(x)$ is not the same as $\nabla ^2f(x)$. Note that the dynamics of Backtracking Levenberg-Marquardt's method is $x\mapsto G(x)=x-\gamma (x)(H(x)^{\intercal}H(x)+\delta _1||F(x)||^{\tau})^{-1}.\nabla f(x)/2$, and it is easy to compute that $G(x)=(H(x)^{\intercal}H(x)+\delta _1||F(x)||^{\tau})^{-1}.\nabla f(x)/2$ is $C^1$ near $x^*$, and $\nabla G(x^*)=(H(x^*)^{\intercal}H(x^*)+\delta _1||F(x^*)||^{\tau})^{-1}.\nabla ^2f(x^*)/2$. Since  $H(x^*)^{\intercal}H(x^*)+\delta _1||F(x^*)||^{\tau}$ is strictly positive definite near $x^*$, it follows that the term $(H(x^*)^{\intercal}H(x^*)+\delta _1||F(x^*)||^{\tau})^{-1}.\nabla ^2f(x^*)/2$ also has at least one negative eigenvalue. Therefore, if $\gamma (x)$ varies $C^1$ near $x^*$ (for example, if it is a constant), then one has a local Stable-Central manifold for $x^*$ for the dynamics of Backtracking Levenberg-Marquardt's method, and then can use the results in the previous Subsection and \cite{truong} to show (when $\delta _0, \delta _1$ are randomly chosen from beginning) that Backtracking Levenberg-Marquardt's algorithm can globally avoid $x^*$. 

There is one way to make $\gamma(x)$ to be constant near $x^*$, that is to choose $\gamma (x)$ by a more sophisticated manner. The main idea in \cite{truong, truongnew}, used for Backtracking line search for gradient descent, is to find a continuous quantity $R(x)$, so that if $\gamma <R(x)$ then Armijo's condition is satisfied. Then, if one chooses $\gamma (x)$ in a sequence $\{\beta ^n: ~n=0,1,2,\ldots \}$, where $0<\beta <1$ is randomly chosen, then one can make sure (in case $x^*$ is an isolated saddle point) that $\gamma (x)$ is constant near $x^*$. In the case at hand, such an $R(x)$ can be bounded from the quantities involving $F(x)$ at $x^*$. The case of non-isolated saddle points can be treated similarly, by using Lindelof's lemma (which is first used in \cite{panageas-piliouras} for the usual Gradient descent method). Therefore, we have the following result. 

\begin{theorem}
If $F:\mathbb{R}^m\rightarrow \mathbb{R}^{m'}$ is $C^2$, and $f=||F||^2$, then by choosing $\delta _0,\delta _1$ randomly, as well as choosing the learning rate $\gamma (x)$ in a sequence $\{\beta ^n: ~n=0,1,2,\ldots \}$ (where $0<\beta <1$ is randomly chosen) by the manner in \cite{truong4, truongnew}, then for a random initial point the sequence $\{x_n\}$ constructed by Backtracking Levenberg-Marquardt's algorithm  cannot converge to a generalised saddle point. 

\label{Theorem8}\end{theorem}   

However, it is still preferable to show avoidance of saddle points for the original version of Backtracking Levenberg-Marquardt's algorithm here, since it is simpler. It is expected that settling this question - at least locally near the saddle point - is the same as settling the question of whether the original version of Armijo's Backtracking line search for Gradient descent can avoid saddle points, and the latter question is still open. (As mentioned, a more sophisticated choice of learning rate for Backtracking line search for Gradient descent can avoid saddle points, see \cite{truong4, truongnew}.) Here, by using the ideas  in the previous Subsection, we can show that Levenberg-Marquardt Backtracking M can avoid saddle points of a special type, which is described next.

Again, let $f(x)=||F(x)||^2$ and $x^*$ a saddle point of $f$. Then $\nabla ^2f(x^*)/2 =H(x^*)^{\intercal}H(x^*)+\nabla ^2F(x^*).F(x^*)$ has at least $1$ negative eigenvalue. Since $H(x^*)^{\intercal}H(x^*)$ is semi-positive definite, the matrix $\nabla ^2F(x^*).F(x^*)$ also has at least one negative eigenvalue. The special generalised saddle points we concern are: 

{\bf Strong generalised saddle points:} $x^*$ is a strong generalised saddle point of $f(x)=||F(x)||^2$ if it is a generalised saddle point of $f$, and moreover $\nabla ^2F(x^*).F(x^*)$ is negative definite.  

\begin{theorem}
Let $F:\mathbb{R}^m\rightarrow \mathbb{R}^{m'}$ be $C^2$ and $f(x)=||F(x)||^2$. Assume that $\delta _0,\delta _1$ are chosen randomly. If $x_0$ is a random initial point, and $\{x_n\}$ is the sequence  constructed from Backtracking Levenberg-Marquardt's algorithm, then $\{x_n\}$ cannot converge to a strong generalised saddle point $x^*$ of $f$.  

\label{Theorem9}\end{theorem}
\begin{proof}
As mentioned before the statement of Theorem \ref{Theorem8}, it suffices to show that for $x$ close to $x^*$, then the learning rate $\gamma (x)=1$.  Denote, as usual,  $w(x)=(H(x)^{\intercal}H(x)+\delta _1||F(x)||^{\tau})^{-1}.H(x)^{\intercal}.F(x)$. Note that $||w(x)||\sim ||H(x)^{\intercal}.F(x)||$ Then, by Taylor's expansion we have
\begin{eqnarray*}
f(x-w(x))-f(x)&=&-2<w(x),H(x)^{\intercal}.F(x)>\\
&&+<[H(x)^{\intercal}H(x)+\nabla ^2F(x).F(x)].w(x),w(x)>+o(||w(x)||^2)
\end{eqnarray*}
 
Since $\nabla ^2F(x).F(x)$ is negative definite when $x$ is close to $x^*$, and since $\delta _1>0$, we obtain 
\begin{eqnarray*}
<[H(x)^{\intercal}H(x)+\nabla ^2F(x).F(x)].w(x),w(x)>&\leq&<[H(x)^{\intercal}H(x)+\delta _1||F(x)||^{\tau }]w(x), w(x)>\\
&=&<H(x)^{\intercal}.F(x),w(x)>.  
\end{eqnarray*}

Hence, for $x$ near $x^*$, we have as wanted: $f(x-w(x))-f(x)\leq - 1/3<w(x),H(x)^{\intercal}.F(x)>$. 
\end{proof}

Note that when $m=1$, then a saddle point of $f$ is also a strong generalised saddle point. On the other hand, for $m=2$, the two notions are different, already for the case $f=|g(z)|^2$ where $g$ is a univariate holomorphic function (see \cite{truong-etal}). 

Concerning avoidance of saddle points, the next result shows that currently Backtracking New Q-Newton's method SE has better theoretical guarantees than Backtracking Levenberg-Marquardt's method. The proof is similar to the proof of Theorem \ref{Theorem9}.

\begin{theorem}
Let $F:\mathbb{R}^m\rightarrow \mathbb{R}^{m'}$ be $C^2$ and $f(x)=||F(x)||^2$. Assume that $\delta _0,\delta _1$ are chosen randomly. If $x_0$ is a random initial point, and $\{x_n\}$ is the sequence  constructed from Backtracking New Q-Newton's method SE, then $\{x_n\}$ cannot converge to a generalised saddle point $x^*$ of $f$.  
\label{Theorem10}\end{theorem}

Like the case of Backtracking New Q-Newton's method, under the assumption on Lojasiewicz inequality (around critical points),  one can prove convergence for Backtracking New Q-Newton's method SE. On the other hand, we do not know if such a convergence can be proven also for Backtracking Levenberg-Macquardt's algorithm. 

\begin{theorem} Assume that $F:\mathbb{R}^m\rightarrow \mathbb{R}^{m'}$ is $C^1$ so that $f=||F||^2$ satisfies the Lojasiewicz gradient inequality. Let $\{x_n\}$ be a sequence constructed by Backtracking New Q-Newton's method SE. Assume also that $0<\tau < 1$. 

1) Assume that for all critical points $x^*$ of $f$, we have $ \mu (x^*)\times (1+\tau ) <1$. Then either $\lim _{n\rightarrow\infty}||x_n||=\infty$ or $\{x_n\}$ converges to a critical point of $f$. 

2) If $F$ is a polynomial map, then the condition in part 1) is satisfied, provided $\tau >0$ is small enough. 
\label{Theorem11}\end{theorem}  
\begin{proof}
The proof is similar to what given for  Backtracking New Q-Newton's method, by using the effective upper bound \cite{acunto-kurdyka} for the Lojasiewicz exponent of a polynomial map (in terms of degree of the map and the dimension). 
\end{proof}

\section{Experimental results on various variants of Newton's method}
\subsection{Implementation details for Backtracking New Q-Newton's method}
\label{SectionImplementation}
In this Subsection, we present some practical points concerning implementation details, for the language Python. Source code is in the GitHub link \cite{tuyenGitHub}. 

Indeed, Python has already enough commands to implement New Q-Newton's method. There is a package, named numdifftools, which allows one to compute approximately the gradient and Hessian of a function. This package is also very convenient when working with a family $f(x,t)$ of functions, where $t$ is a parameter. Another package, named scipy.linalg, allows one to find (approximately) eigenvalues and the corresponding eigenvectors of a square matrix. More precisely, given a square matrix $A$, the command $eig(A)$ will give pairs $(\lambda ,v _{\lambda})$ where $\lambda $ is an approximate eigenvalue of $A$ and $v$ a corresponding eigenvector. 

One point to notice is that even if $A$ is a symmetric matrix with real coefficients, the eigenvalues computed by the command $eig$ could be complex numbers, and not real numbers, due to the fact that these are approximately computed. This can be easily resolved by taking the real part of $\lambda$, which is given in Python codes by $\lambda .real$. We can do similarly for the eigenvectors. A very convenient feature of the command $eig$ is that it already computes (approximate) orthonormal bases for the eigenspaces. Note that similar commands are also available on PyTorch and TensorFlow, two popular libraries for implementing Deep Neural Networks.   

Now we present the coding detail of the main part of New Q-Newton's method: Given a symmetric invertible matrix $A$ with real coefficients (in our case $A=\nabla ^2f(x_n)$ $+\delta _j||\nabla f(x_n)||^{1+\alpha}$), and a vector $v$, compute $w$ which is the reflection of $A^{-1}.v$ along the direct sum of eigenspace of negative eigenvectors of $A$. First, we use the command $eig$ to get pairs $\{(\lambda _j, v_j)\}_{j=1,\ldots ,m}$. Use the command real to get real parts. If we write $v=\sum _{j=1}^m a_jv_j$, then $a_j=<v_j,v>$ (the inner product), which is computed by the Python command $np.dot(v_j,v)$. Then $v_{inv}=A^{-1}v=\sum _{j=1}^m(a_j/\lambda _j)v_j$. Finally, 
\begin{eqnarray*}
w=v_{inv}-2\sum _{j:~\lambda _j<0}(a_j/\lambda _j)v_j.
\end{eqnarray*}

\begin{remark} The implementation of Backtracking New Q-Newton's method (and other versions introduced in this paper) is very easy, flexible and works well and stably. 
 
1) We do not need to compute exactly the gradient and the Hessian of the cost function $f$, only approximately. Indeed, if one wants to stop when $||\nabla f(x_n)||$ and $||x_n-x_{\infty}||$ is smaller than a threshold $\epsilon$, then it suffices to compute the gradient and the Hessian up to an accuracy of order $\epsilon$. 

Similarly, we do not need to compute the eigenvalues and eigenvectors of the Hessian exactly, but only up to an accuracy of order $\epsilon$, where $\epsilon$ is the threshold to stop. 

In many experiments,  we only calculate the Hessian inexactly using the numdifftools package in Python, and still obtain good performance. 

2) While theoretical guarantees are proven only when the hyperparameters $\delta _0,\ldots ,\delta _m$ are randomly chosen and fixed from the  beginning, in experiments we have also tried to choose - at each iterate $n$ - choose randomly a $\delta$. We find that this variant has {a similar or better performance} as the original version. 

3) Similarly, the exponent $\tau$ in the algorithms can be flexibly chosen, need not to be $\leq 1$ as required in theoretical results. 

4) We simplify further by not checking the condition $minsp(\nabla^2f(x_k)+\delta_j \|\nabla f(x_k)\|^{\tau}Id)\geq \kappa  \|\nabla f(x_k)\|^{\tau}$, but only if $\nabla^2f(x_k)+\delta_j \|\nabla f(x_k)\|^{\tau}Id$ is invertible or not. Experiments show that this version behaves similarly to the official version of Backtracking New Q-Newton's method. 

Another simplification is that instead of choosing $m+1$ numbers $\delta _0,\ldots ,\delta _m$, we choose only 2 numbers $\delta _0$ and $\delta _1$ so that $|\delta _1-\delta _0|$ is small. Indeed, we choose $\delta _0=1.00001$ and $\delta _2=0.999$. 

\end{remark}

\subsection{Experimental results }

Here we present a couple of illustrating experimental results. Additional experiments, which are quite extensive, will be presented in the arXiv version of the paper. We use the python package numdifftools \cite{num} to compute gradients and Hessian, since symbolic computation is not quite efficient. The experiments here are run on a small personal laptop. The unit for running time is seconds.

Here, we will compare the performance of Backtracking New Q-Newton's method (and its simplified version Simplified Backtracking New Q-Newton's method) against several, including well known, existing variants of Newton's method:  the usual Newton's method, BFGS, Adaptive Cubic Regularization \cite{nesterov-polyak} \cite{cartis-etal}, as well as Random damping Newton's method \cite{sumi}  and Inertial Newton's method \cite{bolte-etal}.
 
For Backtracking New Q-Newton's method and Simplified Backtracking Newton's method,  we choose $\tau =1$ in the definition. Moreover, we will choose $\Delta =\{0,\pm 1\}$, even though for theoretical proofs we need $\Delta$ to have at least $m+1$ elements, where $m=$ the number of variables. The justification is that when running Backtracking New Q-Newton's method it almost never happens the case that both $\nabla ^2f(x)$ and $\nabla ^2f(x)\pm ||\nabla f(x)||^2Id$ are not invertible. The experiments are coded in Python and run on a usual personal computer. For BFGS: we use the function scipy.optimize.fmin$\_$bfgs available in Python, and put  $gtol=1e-10$ and $maxiter=1e+6$. For Adaptive cubic regularization for Newton's method, we use the AdaptiveCubicReg module in the implementation in \cite{ARCGitHub}. We use the default hyperparameters as recommended there, and use ``exact" for the hessian$\_$update$\_$method.  For hyperparameters in Inertial Newton's method, we choose $\alpha =0.5$ and $\beta =0.1$ as recommended by the authors of \cite{bolte-etal}. Source codes for the current paper are available at the GitHub link \cite{tuyenGitHub}. For Levenberg-Macquardt's method, we use $\lambda _n=||f(x_n)||^{\tau}$, with the same exponent $\tau=1$ as in Backtracking New Q-Newton's method. For Regularized Newton's method, we choose $c_1=1.5$ and $\lambda _n=||\nabla f(x_n)||^{\tau}$. 

Except one, in all of the experiments, we choose simply $\widehat{w_k}=w_k$ in Backtracking line search (in the algorithms Backtracking Regularized Newton's method, Backtracking Levenberg-Marquardt's method, Backtracking New Q-Newton's method, and Simplified Backtracking New Q-Newton's method), even if the cost function does not have compact sublevels. 

{\bf Feature reported:} We report on the number of iterates needed for an algorithm to achieve a prescribed threshold (more precisely, if $||\nabla f(z_n)||<\epsilon $ or $||z_{n+1}-z_n||<\epsilon $, where we fix $\epsilon =10^{-10}$ in the experiments) or the number of iterates reaches a threshold (10000 iterations), the time (in seconds) it needs to run, the values of the cost function and its gradient at the end point (for a system of equations, we also report the value of each function in the system), whether the Hessian of the cost function at the end point has a negative eigenvalue. 

{\bf Legends:} We use the following abbreviation in the reports: ``ACR" for Adaptive Cubic Regularization,   ``BFGS" for BFGS, ``Newt" for Newton's method, ``Newt-SE" for Newton's method applied directly to a system of equations $F$ instead of to the associated cost function $f=||F||^2/2$,  ``RegN" for  Regularized Newton's method, ``B-RegN" for Regularized Newton's method with Backtracking line search incorporated,  ``LM" for Levenberg-Marquardt method, ``B-LM" for Levenberg-Marquardt method with Backtracking line search incorporated,  ``Iner" for Inertia Newton's method, ``B-NewQ" for Backtracking New Q-Newton's method,  and ``SB-NewQ" for Simplified Backtracking New Q-Newton's method. 

\subsubsection{One polynomial system in 2 real variables}

This example is taken from \cite{fr}. We consider the following system in 2 real variables $x_1,x_2$. 
\begin{eqnarray*}
f_1:=p_{10}+p_{11}x_1+p_{12}x_2+p_{13}x_2^2+p_{14}x_2^3&=&0,\\
f_2:=p_{20}+p_{21}x_1+p_{22}x_2+p_{23}x_2^2+p_{24}x_2^3&=&0,
\end{eqnarray*}
where $p_{10}=-13$, $p_{11}=1$, $p_{12}=-2$, $p_{13}=5$, $p_{14}=-1$, $p_{20}=-29$, $p_{21}=1$, $p_{22}=-14$, $p_{23}=1$ and $p_{24}=1$. Recall that $f=(f_1^2+f_2^2)/2$. We chose randomly initial points with coordinates in the interval $[-100,100]$. 

For the first experiment, the initial point is $z_0=[-9.12027123, -3.7284278 ]$. It turns out that only Newt-SE converges to a global minimum of the cost function $f$ ($[5,4]$), the other methods either converge to a local minimum or (ACR, Iner) stopped before being close enough to a critical point. See Table \ref{tab:Experiment1} for details. 

\begin{table}[htp]
\fontsize{11}{11}\selectfont
  \centering

\begin{tabular}{c|c|c|c|c|c|c|c}
& $\sharp N$ &$f(z_N)$ &$||\nabla f(z_N)||$&$f_1(z_N)$ &$f_2(z_N)$&eigenS& Time   \\ \hline
ACR&6&334&42& 3.77&-25&1.02&0.02 \\
BFGS &15&24.49&4e-12&4.9&-4.9&4e-1&0.13\\
Newt &10&24.49&8e-14&4.9&-4.9&4e-1&0.10 \\
Newt-SE &16&6e-28& 1e-12&2e-14&-2e-14&1.4&0.19\\
RegN &19&24.49&1e-14&4.9&-4.9&4e-1&0.19\\
B-RegN & 20&24.49&4e-14&4.9&-4.9&4e-1&0.26\\
LM & 58&24.49&3e-10&4.9&-4.9&4e-1&0.69\\
B-LM &46&24.49&6e-7&4.9&-4.9&4e-1&0.55 \\
Iner &2&7e+85&2e+72&8e+42&-8e+42&0&0.01 \\ \hline
B-NewQ &10&24.49&3e-10&4.9&-4.9&4e-1&0.10\\
SB-NewQ &19&24.49&1e-13&4.9&-4.9&4e-1&0.21
 \end{tabular}
   \caption{Experimental results  using various variants of Newton's method to  solve the system of equations $f_1:=-13+1x_1+2x_2+5x_2^2-1x_2^3=0$ and $f_2:=-29+1x_1+14x_2+1x_2^2+1x_2^3=0$ in 2 real variables $x_1$ and $x_2$.  The cost function is $f=(f_1^2+f_2^2)/2$. The initial point is $z_0=[-84.439842, -1.60847421 ]$, at which $f_1(z_0)\sim -77$, $f_2(z_0)\sim -92$, and $f(z_0)\sim 7251$. {\bf Legends}: $\sharp N$ (maximum $10000$) =is the number of iterates the concerned algorithm runs before stopping when $||\nabla f(z_N)||$ or $||z_N-z_{N-1}||$ is small than a threshold or when some errors happen, $z_N=$ the point constructed at the step $N$, ``eigenS"=the smallest eigenvalue of the Hessian of $f$ at $z_N$, and ``Err" means that some errors (e.g. NAN, overflow, and so on) happened.}
  \label{tab:Experiment1}
\end{table}

For the second and third experiments, we choose the initial points $z_0=[-84.439842, -1.60847421]$ and $z_0=[15,-2]$. The results obtained are similar to that in Table \ref{tab:Experiment1}. In the next Subsubsection, we consider the same system but now with complex variables. 

\subsubsection{One polynomial system in 2 complex variables}

Here, we explore the same system as in the previous Subsubsection: 
\begin{eqnarray*}
f_1:=p_{10}+p_{11}x_1+p_{12}x_2+p_{13}x_2^2+p_{14}x_2^3&=&0,\\
f_2:=p_{20}+p_{21}x_1+p_{22}x_2+p_{23}x_2^2+p_{24}x_2^3&=&0,
\end{eqnarray*}
where $p_{10}=-13$, $p_{11}=1$, $p_{12}=-2$, $p_{13}=5$, $p_{14}=-1$, $p_{20}=-29$, $p_{21}=1$, $p_{22}=-14$, $p_{23}=1$ and $p_{24}=1$. The difference here is that we consider $x_1,x_2$ as {\bf complex} variables, and hence the cost function $f$ is modified to $f=(|f_1|^2+|f_2|^2)/2$. It means that we are finding roots to 4 equations in 4 real variables. 

We will choose initial points as small perturbations of the initial points chosen in the previous Subsubsection.

For the first experiment, we choose the initial point $$z_0=[-9.12027123+ 0.001 i, -3.7284278 - 0.001 i].$$  In this case, it turns out that Newt-SE still converges to the same global minimum $[5,4]$ as in the real case, Newton's method still converges to the same point as it did in the real case but now the point is no longer a local minimum (!),  ACR and Iner are stopped before close enough to a critical point, while other methods converge to a different global minimum ($[13-14i, -1-1i]$). See Table \ref{tab:Experiment2} for details. 

\begin{table}[htp]
\fontsize{11}{11}\selectfont
  \centering

\begin{tabular}{c|c|c|c|c|c|c|c}
& $\sharp N$ &$f(z_N)$ &$||\nabla f(z_N)||$&$|f_1(z_N)|$ &$|f_2(z_N)|$&eigenS& Time   \\ \hline
ACR&3&3e+93&5e+78& 5e+46&5e+46&0.00&0.02 \\
BFGS &27&1e-26&3e-12&1e-13&4e-14&6e-1&0.30\\
Newt &10&24.49&2e-14&4.9&4.9&-6e-10&0.10 \\
Newt-SE &16&2e-28&$1e-14$ &1e-14&-1e-14&1.4&0.76\\
RegN &19&24.49&1e-14&4.9&-4.9&4e-1&0.19\\
B-RegN & 20&24.49&4e-14&4.9&-4.9&4e-1&0.26\\
LM & 31&6e-30&3e-14&2e-15&2e-15&6e-1&1.50\\
B-LM &33&9e-30&7e-14&2e-15&3e-15&6e-1&1.59 \\
Iner &2&7e+85&2e+72&8e+42&-8e+42&0.00&0.01 \\ \hline
B-NewQ &31&5e-27&1e-13&6e-14&7e-14&6e-1&1.30\\
SB-NewQ &25&8e-27&1e-13&8e-14&1e-12&6e-1&1.18
 \end{tabular}
   \caption{Experimental results  using various variants of Newton's method to  solve the system of equations $f_1:=-13+1x_1+2x_2+5x_2^2-1x_2^3=0$ and $f_2:=-29+1x_1+14x_2+1x_2^2+1x_2^3=0$ in 2 complex variables $x_1$ and $x_2$.  The cost function is $f=(|f_1|^2+|f_2|^2)/2$. The initial point is $z_0=[-9.12027123+ 0.001 i, -3.7284278 - 0.001 i]$, at which $|f_1(z_0)|\sim 106$, $|f_2(z_0)|\sim 23.8$, and $f(z_0)\sim 5971$. {\bf Legends}: $\sharp N$ (maximum $10000$) =is the number of iterates the concerned algorithm runs before stopping when $||\nabla f(z_N)||$ or $||z_N-z_{N-1}||$ is small than a threshold or when some errors happen, $z_N=$ the point constructed at the step $N$, ``eigenS"=the smallest eigenvalue of the Hessian of $f$ at $z_N$, and ``Err" means that some errors (e.g. NAN, overflow, and so on) happened.}
  \label{tab:Experiment2}
\end{table}

For the second experiment, we choose the initial point $z_0=[15-0.000001i, -2-0.00002i]$, which is a small perturbation of the corresponding point in the real case. We obtain similar results like in Table \ref{tab:Experiment2}, except that the performance of ACR is improved: after 2 iterates, it stopes (because of some ill-conditioned Hessian errors) at a point whose function value is $f\sim 4230$. 

For the third experiment, we choose the initial point $z_0=[-84.439842 -0.000001i,    -1.60847421+0.000001 i]$, which is a small perturbation of the corresponding point in the real case. We obtain similar results like in Table \ref{tab:Experiment2}, except that the performance of ACR is improved: after 12 iterates, it stopes (because of some ill-conditioned Hessian errors) at a point close to $\sim [13+14i,-1+1i]$ which is another global minimum of $f$.  

\subsubsection{One non-polynomial system in 3 real variables} We consider the following system, taken from \cite{hueso-etal} , 
\begin{eqnarray*}
f_1:=3x_1-cos(x_2x_3)-0.5&=&0,\\
f_2:=x_1^2-625x_2^2-0.25&=&0,\\
f_3:=e^{-x_1x_2}+20x_3+(10\pi -3)/3&=&0,
\end{eqnarray*}
in 3 real variables $x_1,x_2,x_3$. This system has a root at $z^*=(0.5,0,-\pi /6 )\sim (0.5,0,-0.523)$. At this point, the Jacobian of the system is singular.

Recall that $f=(f_1^2+f_2^2+f_3^2)/2$. We chose randomly initial points with coordinates in the interval $[-50,50]$.

For the first experiment, we choose the initial point to be 
$$z_0=[-42.38817886, -13.88913045,  10.93977723].$$ See Table \ref{tab:Experiment3} for details. In this case, since the Backtracking Regularized Newton's method stops long before the maximum number of iterates (10000), and the Hessian of the function at the stopping point has a negative eigenvalue, it seems that this method cannot avoid saddle points.

\begin{table}[htp]
\fontsize{11}{11}\selectfont
  \centering

\begin{tabular}{c|c|c|c|c|c|c|c|c}
& $\sharp N$ &$f(z_N)$ &$||\nabla f(z_N)||$&$f_1(z_N)$ &$f_2(z_N)$&$f_3(z_N)$&eigenS& Time   \\ \hline
ACR&Err&&&&&&& \\
BFGS &97&3e-20&6e-11&7e-11&-2e-10&3e-12&8e-7&0.58\\
Newt &72&3e-21&2e-11&2e-11&-8e-11&5e-15&3e-7&1.45 \\
Newt-SE &Err&&&&&&\\
RegN &247&6e-21&2e-11&2e-11&-1e-10&8e-15&4e-7&5.40\\
B-RegN & 215&404&697&27.6&-4.95&-4.54&-31.1&4.74\\
LM & Err&&&&&&\\
B-LM &62&1e-21&5e-11&-5e-14&-5e-11&8e-15&1e-7&1.40 \\
Iner &Err&&&&&&& \\ \hline
B-NewQ &35&4e-21&2e-11&2e-11&-9e-11&5e-15&3e-7&0.59\\
SB-NewQ &38&7e-21&3e-11&2e-11&-1e-10&7e-15&4e-7&0.605
 \end{tabular}
   \caption{Experimental results  using various variants of Newton's method to  solve the system of equations $f_1:=3x_1-cos(x_2x_3)-0.5=0$,  $f_2:=x_1^2-625x_2^2-0.25=0$ and $f_3:=e^{-x_1x_2}+20x_3+(10\pi -3)/3=0$ in 3 real variables $x_1$, $x_2$ and $x_3$.  The cost function is $f=(f_1^2+f_2^2+f_3^2)/2$. The initial point is $z_0=[-42.38817886, -13.88913045,  10.93977723]$, at which $f_1(z_0)\sim -128$, $f_2(z_0)\sim -1e+5$, $f_3(z_0)=228$ and $f(z_0)\sim 7e+9$. {\bf Legends}: $\sharp N$ (maximum $10000$) =is the number of iterates the concerned algorithm runs before stopping when $||\nabla f(z_N)||$ or $||z_N-z_{N-1}||$ is small than a threshold or when some errors happen, $z_N=$ the point constructed at the step $N$, ``eigenS"=the smallest eigenvalue of the Hessian of $f$ at $z_N$, and ``Err" means that some errors (e.g. NAN, overflow, and so on) happened.}
  \label{tab:Experiment3}
\end{table}

For the second experiment, we choose the initial point to be 
$$z_0=[-42.68403992, -47.90598209,  22.59078781].$$ See Table \ref{tab:Experiment4} for details. In this case, since the Backtracking Regularized Newton's method stops long before the maximum number of iterates (10000), and the Hessian of the function at the stopping point has a negative eigenvalue, it seems that this method cannot avoid saddle points.

\begin{table}[htp]
\fontsize{11}{11}\selectfont
  \centering

\begin{tabular}{c|c|c|c|c|c|c|c|c}
& $\sharp N$ &$f(z_N)$ &$||\nabla f(z_N)||$&$f_1(z_N)$ &$f_2(z_N)$&$f_3(z_N)$&eigenS& Time   \\ \hline
ACR&Err&&&&&&& \\
BFGS &Err&&&&&&&\\
Newt &41&1e-20&4e-11&3e-11&-1e-10&8e-15&5e-7&0.89 \\
Newt-SE &Err&&&&&&\\
RegN &420&5e-21&2e-11&2e-11&-1e-10&1e-15&3e-7&9.17\\
B-RegN &387&394&813&27.4&-4.46&-3.74&-2.43&8.26\\
LM & Err&&&&&&\\
B-LM &Err&&&&&&& \\
Iner &Err&&&&&&& \\ \hline
B-NewQ &39&6e-21&2e-11&2e-11&-1e-10&5e-15&4e-7&0.58\\
SB-NewQ &42&3e-21&2e-11&1e-11&-7e-11&1e-15&2e-7&0.61
 \end{tabular}
   \caption{Experimental results  using various variants of Newton's method to  solve the system of equations $f_1:=3x_1-cos(x_2x_3)-0.5=0$,  $f_2:=x_1^2-625x_2^2-0.25=0$ and $f_3:=e^{-x_1x_2}+20x_3+(10\pi -3)/3=0$ in 3 real variables $x_1$, $x_2$ and $x_3$.  The cost function is $f=(f_1^2+f_2^2+f_3^2)/2$. The initial point is $z_0=[-42.68403992, -47.90598209,  22.59078781]$, at which $f_1(z_0)\sim -128$, $f_2(z_0)\sim -1e+6$, $f_3(z_0)=461$ and $f(z_0)\sim 1e+12$. {\bf Legends}: $\sharp N$ (maximum $10000$) =is the number of iterates the concerned algorithm runs before stopping when $||\nabla f(z_N)||$ or $||z_N-z_{N-1}||$ is small than a threshold or when some errors happen, $z_N=$ the point constructed at the step $N$, ``eigenS"=the smallest eigenvalue of the Hessian of $f$ at $z_N$, and ``Err" means that some errors (e.g. NAN, overflow, and so on) happened.}
  \label{tab:Experiment4}
\end{table}

\section{Extensions, Some pictures of basins of attraction, and Conclusions}

\subsection{Some further extensions} We can extend the algorithms above to the setting where at each iterate one projects to subspaces different from eigenspaces of the Hessian of $f$, or more generally working with a symmetric square matrix having no relations to the Hessian of $f$. This way, we can unify first and second order methods, and among these algorithms some have the flavour of quasi-Newton's methods. 

\medskip
{\color{blue}
 \begin{algorithm}[H]
\SetAlgoLined
\KwResult{Find a minimum of $f:\mathbb{R}^m\rightarrow \mathbb{R}$}
Given: $\{\delta_0,\delta_1,\ldots, \delta_{m}\} \subset \mathbb{R}$\ (chosen {\bf randomly}), $0<\tau $, $q\geq 1$ and $0<\gamma _0<1$;\\
Initialization: $x_0\in \mathbb{R}^m$\;
$\kappa:=\frac{1}{2}\min _{i\not=j}|\delta _i-\delta _j|$;\\
 \For{$k=0,1,2\ldots$}{ 
    $j=0$\\
  \If{$\|\nabla f(x_k)\|\neq 0$}{Choose a symmetric real matrix $B_k$\;
   \While{$minsp(B_k+\delta_j \|\nabla f(x_k)\|^{\tau}Id)<\kappa  \|\nabla f(x_k)\|^{\tau}$}{$j=j+1$}}
  
 $A_k:=B_k+\delta_j \|\nabla f(x_k)\|^{\tau}Id$\\
$e_{1,k},\ldots ,e_{m,k}$ an appropriately chosen orthonormal basis for $\mathbb{R}^m$\\
$w_k:=\sum _{i=1}^m\frac{<\nabla f(x_k),e_{i,k}>}{(\sum _{j=1}^m|<A_k.e_{i,k},e_{j,k}>|^q)^{1/q}}e_{i,k}$\\
$\widehat{w_k}:=w_k/\max \{1,||w_k||\}$\\
$\gamma :=\gamma _0$\\
 \If{$\|\nabla f(x_k)\|\neq 0$}{
   \While{$f(x_k-\gamma \widehat{w_k})-f(x_k)>-\gamma <\widehat{w_k},\nabla f(x_k)>/3$}{$\gamma =\gamma /3$}}

$x_{k+1}:=x_k-\gamma \widehat{w_k}$
   }
  \caption{Backtracking New Q-Newton's method G} \label{table:algG}
\end{algorithm}
}
\medskip

First, we investigate some special cases of Backtracking New Q-Newton's method G where $B_k:=\nabla ^2f(x_k)$, and where $q=2$ (in which case $(\sum _{j=1}^m|<A_k.e_{i,k},e_{j,k}>|^q)^{1/q}=||A_ke_{i,k}||$ the usual Euclidean norm of the vector $A_ke_{i,k}$). We fix $e_1,\ldots ,e_m$ an orthonormal basis for $\mathbb{R}^m$. 

{\bf Backtracking New Q-Newton's method:} We choose $e_{1,k},\ldots ,e_{m,k}$ eigenvectors of $A_k$. Hence, if $\lambda _i$ is the corresponding eigenvalue of $e_{i,k}$ of $A_k$, then 
\begin{eqnarray*}
w_k=\sum _{i=1}^m\frac{<\nabla f(x_k),e_{i,k}>}{|\lambda _i|}e_{i,k}. 
\end{eqnarray*}

{\bf Backtracking New Q-Newton's method Backtracking G1:}  We choose $e_1(x),\ldots ,e_m(x)$'s to be eigenvectors of $\nabla ^2f(x)$ if $minsp(\nabla ^2f(x))\geq ||\nabla f(x)||^{1/2}$, otherwise $e_1(x)$, $\ldots$, $e_m(x)$ are $e_1$, $\ldots$ , $e_m$. This is a combination between New Q-Newton's method Backtracking and New Q-Newton's method Backtracking G2 (see next). As such, it has fast rate of convergence as New Q-Newton's method Backtracking, while also has good convergence guarantee as New Q-Newton's method Backtracking G3. 

{\bf New Q-Newton's method Backtracking G2:} We choose $e_{i,k}=e_i$ for all $i=1,\ldots ,m$. This version is probably the less computationally expensive, and hence has some flavours of quasi-Newton's methods. On the other hand, it has quite good theoretical guarantees. 

\begin{remark}

Usually, the choice of $\tau >1$ will assure that the associated dynamical system is $C^1$, which helps more with avoidance of saddle points. 

On the other hand, the choice of $0<\tau <1$ helps more with theoretical convergence guarantees. 

However, in experiments, we found that there is no real difference between the choices $\tau >1$ or $0<\tau <1$ or $\tau =1$. 

\end{remark}

As in Section 2, the following result can be used to establishing that Backtracking New Q-Newton's method G can avoid saddle points. (The convergence of it can be dealt similarly to that for Backtracking New Q-Newton's method and other methods in Section 2.)

\begin{theorem} Let $f(x)=||F(x)||^2$ be as above. Let $w(x)$ be chosen as in the algorithm Backtracking New Q-Newton's method G. Assume that $q\geq 1$ is so that its H\"older's conjugate $p$ (i.e. the number $p\geq 1$ so that $p^{-1}+q^{-1}=1$, the value $+\infty$ is allowed) satisfies $m^{1/p}<4/3$. Assume moreover that for $x$ close enough to a generalised saddle point $x^*$ of $f$, we have $B(x)-\nabla ^2f(x)+O(||\nabla f(x)||)$ is positive definite. Then, for all $x$ close enough to $x^*$ we have $\gamma (x)=1$. 

\label{Theorem7}\end{theorem}
\begin{proof}
Since $||F(x)||$ is bounded away from 0 near a generalised saddle point $x^*$ of $f$, we have that $||w(x)||\sim ||\nabla f(x)||$ near $x^*$. Also, $<w(x),\nabla f(x)>\sim ||w(x)||^2$.   By Taylor's expansion we have
\begin{equation}
f(x-w(x))-f(x)=-<w(x),\nabla f(x)>+\frac{1}{2}<\nabla ^2f(x)w(x),w(x)>+O(||w(x)||^3). 
\label{Equation11}\end{equation}
Denote $a_i=<\nabla f(x),e_i(x)>$, $b_{i,j}=<A(x)e_i(x),e_j(x)>$, and $B_i=(\sum _{j=1}^m|b_{i,j}|^q)^{1/q}$. We have
\begin{eqnarray*}
<w(x),\nabla f(x)>=\sum _{i=1}^ma_i^2/B_i. 
\end{eqnarray*}

 Now we bound from above the second summand in the RHS in Equation (\ref{Equation11}), note that all $\delta _j>0$: 
\begin{eqnarray*}
<\nabla ^2f(x)w(x),w(x)>&\leq& <[B(x)+O(||w(x)||].w(x),w(x)>\\
&\leq& <A(x)w(x),w(x)>+O(||w(x)||^3)\\
&=&<\sum _{i=1}^ma_iA(x)e_i(x)/B_i,\sum _{j=1}^ma_je_j(x)/B_j>\\
&=&\sum _{i=1}^m\sum _{j=1}^ma_ia_jb_{i,j}/(B_iB_j)
\end{eqnarray*}
Now, by Cauchy-Schwartz inequality, we have
\begin{eqnarray*}
|a_ia_jb_{i,j}|/(B_iB_j) \leq \frac{a_i^2|b_{i,j}|}{2B_i^2}+\frac{a_j^2|b_{i,j}|}{2B_j^2}
\end{eqnarray*}
By H\"older's inequality, we obtain for all $i$:
\begin{eqnarray*}
\sum _{j=1}^m|b_{i,j}|\leq m^{1/p}B_i
\end{eqnarray*}
This implies, from the assumption on $p$, that
\begin{eqnarray*}
f(x-w(x))-f(x)<-\frac{1}{3}<w(x),\nabla f(x)>
\end{eqnarray*}
for all $x$ close enough to $x^*$. Thus we can choose $\gamma (x)=1$ for all such points. 
\end{proof}

\subsection{Some pictures of basins of attraction}

Below, we display pictures of basins of attractions when applying Backtracking New Q-Newton's method  for some polynomials of small degrees (2,3 4 and 5), and a transcendental function. We compare these with the basins obtained by the standard Newton's method and Backtracking Gradient Descent. It is observed that the basins for Backtracking New Q-Newton's method are much more regular than those for the standard Newton's method and Backtracking Gradient Descent, and hence the regularity observed does not come from properties of Backtracking Gradient Descent alone. 

The pictures below are created by choosing the initial point $z_0$ in a lattice $v+( 0.1j, 0.1k )$, for $j,k\in [-30,30]$, and $v$ is a randomly chosen vector with coordinates in $[-1.1]$.

\underline{\bf For a polynomial of degree 2:} 

We consider the polynomial $P_2(z)=(z-z_1^*)(z-z_2^*)$, where $z_1^*=0.5-0.2i$ and $z_2^*=1+0.4i$. This polynomial satisfies the conditions for to apply the results in \cite{truong-etal}, and hence when we apply Backtracking New Q-Newton's method for the function $f(x,y)=\frac{1}{2}|P(x+iy)|^2$,  if the initial point $z_0$ is randomly chosen then the constructed sequence $\{z_n\}$ will converge, to either $z_1^*$ or $z_2^*$. The basins of attraction (see  Figure \ref{fig:Degree2}) seem to very much satisfy the mentioned result by Arthur Cayley.

\begin{figure}
\centering
        \begin{subfigure}[b]{0.5\textwidth}
        
         \includegraphics[width=\linewidth ]{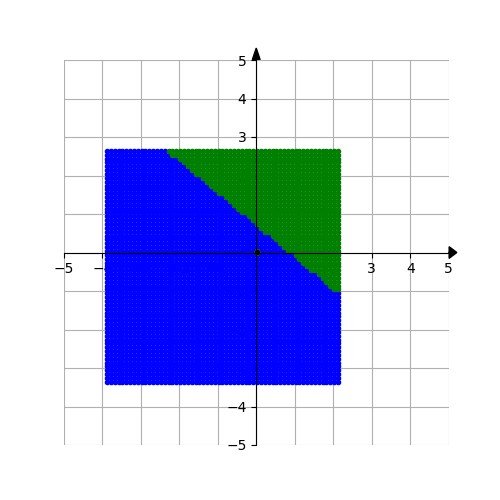}
              \caption{For Backtracking Gradient Descent.} 
                      \end{subfigure}%
        \begin{subfigure}[b]{0.5\textwidth}
              
               \includegraphics[width=\linewidth ]{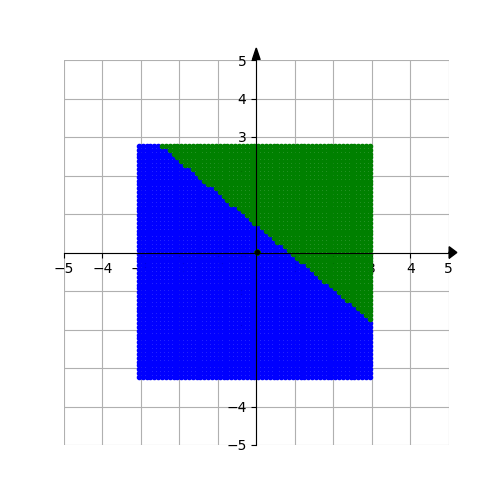}
              \caption{For Backtracking New Q-Newton's method. }

                      \end{subfigure}%
        \caption{Basins of attraction when applying Backtracking New Q-Newton's method to the polynomial $P_2(z)$.   Blue: initial points $z_0$ for which the constructed sequence converges to $z_1^*$.  Green: region consists of initial points $z_0$ for which the constructed sequence converges to $z_2^*$. It seems that the theorem of Arthur Cayley (for Newton's method) is also valid for Backtracking Gradient Descent and Backtracking New Q-Newton's method.} 
       \label{fig:Degree2}
\end{figure}

\underline{\bf For a polynomial of degree 3:} 

We consider the polynomial $P_3(z)=z^3-2z+2$, which has 3 roots: $z_1^*\sim -1.76929$, $z_2^*\sim 0.884646-0.589743i$ and $z_3^*=0.884646+0.589743i$. It is known \cite{NewtonFractal}, that when applying the standard Newton's method to this polynomial, then the basins of attraction are fractal. Moreover, there are sets of positive Lebesgue measure where Newton's method applied to an initial point in these sets will not converge to any of the roots. 

In contrast, since this polynomial satisfies the conditions to apply the results in \cite{truong-etal}, Backtracking New Q-Newton's method applied to  a random initial point $z_0$ will always converge to one of the 3 roots. Moreover, the basins of attraction do not seem to have a fractal structure.  

Basins of attraction for Backtracking Gradient Descent seem to be more regular than that for Newton's method, but less regular than that for Backtracking New Q-Newton's method. 

\begin{figure}
\centering

	 \begin{subfigure}[b]{0.7\textwidth}
        
         \includegraphics[width=\linewidth ]{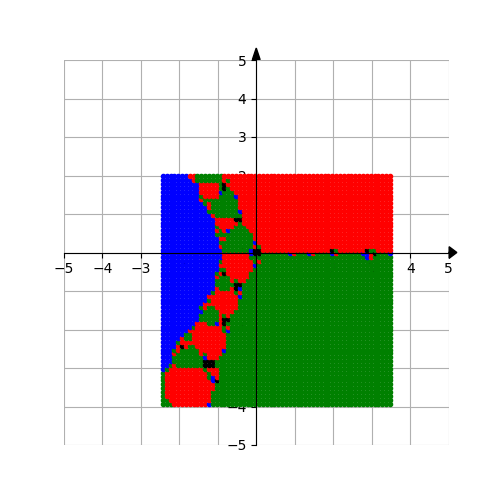}
              \caption{For Newton's method.} 
                      \end{subfigure}%

        \begin{subfigure}[b]{0.5\textwidth}
        
         \includegraphics[width=\linewidth ]{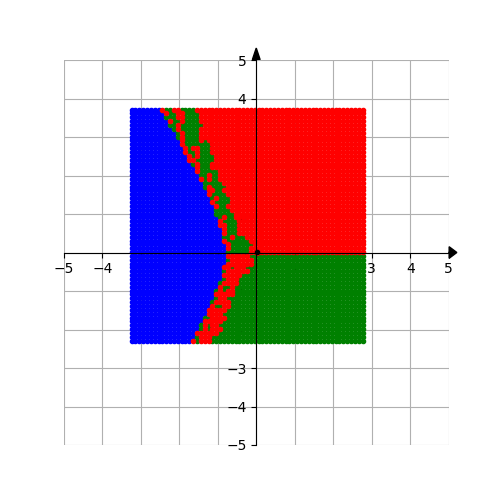}
              \caption{For Backtracking Gradient Descent.} 
                      \end{subfigure}%
        \begin{subfigure}[b]{0.5\textwidth}
              
               \includegraphics[width=\linewidth ]{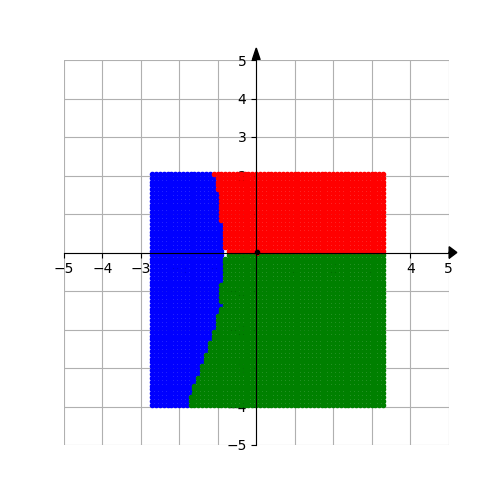}
              \caption{For Backtracking New Q-Newton's method. }

                      \end{subfigure}%
        \caption{Basins of attraction for the polynomial $P_3(z)=z^3-2z+2$. Blue: initial points $z_0$ for which the constructed sequence converges to $z_1^*$.  Green: similar points for the root $z_2^*$. Red: similar points for the root $z_3^*$. Black:  other points.} 
       \label{fig:Degree3}
\end{figure}

\underline{\bf For a polynomial of degree 4:} 

We consider the polynomial $P_4(z)=(z^2+1)(z-2.3)(z+2.3)$, which has 4 roots: $z_1^*=2.3$, $z_2^*=-2.3$, $z_3^*=i$ and $z_4=-i$. When applying the standard Newton's method to this polynomial, then the basins of attraction are fractal. Moreover, there are sets of positive Lebesgue measure where Newton's method applied to an initial point in these sets will not converge to any of the roots. 

In contrast, since this polynomial satisfies the conditions to apply the results in \cite{truong-etal}, Backtracking New Q-Newton's method applied to  a random initial point $z_0$ will always converge to one of the 3 roots. Moreover, the basins of attraction do not seem to have a fractal structure.  

Basins of attraction for Backtracking Gradient Descent seem to be more regular than that for Newton's method, but less regular than that for Backtracking New Q-Newton's method.

\begin{figure}
\centering

	 \begin{subfigure}[b]{0.5\textwidth}
        
         \includegraphics[width=\linewidth ]{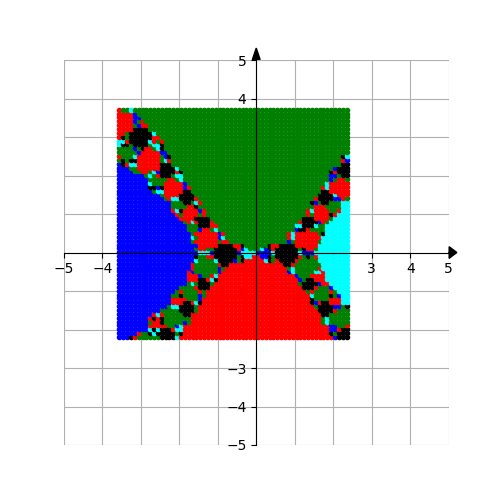}
              \caption{For Newton's method.} 
                      \end{subfigure}%

        \begin{subfigure}[b]{0.5\textwidth}
        
         \includegraphics[width=\linewidth ]{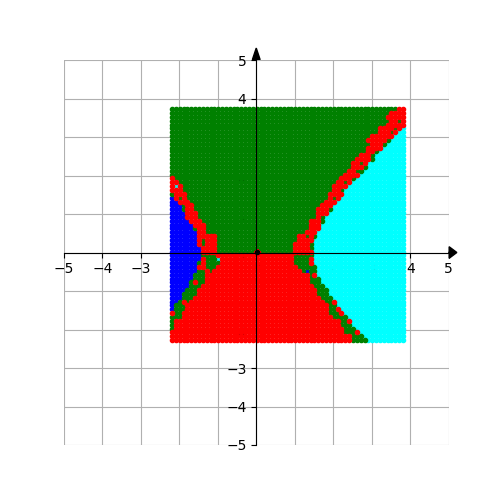}
              \caption{For Backtracking Gradient Descent.} 
                      \end{subfigure}%
        \begin{subfigure}[b]{0.5\textwidth}
              
               \includegraphics[width=\linewidth ]{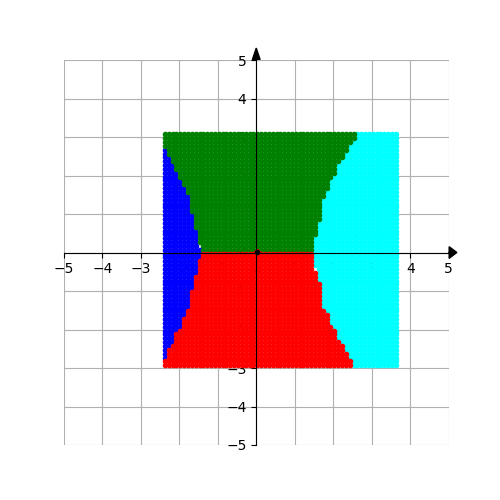}
              \caption{For Backtracking New Q-Newton's method. }

                      \end{subfigure}%
        \caption{Basins of attraction for the polynomial $P_4(z)=(z^2+1)(z-2.3)(z+2.3)$. Blue: initial points $z_0$ for which the constructed sequence converges to $z_1^*$.  Cyan: similar points for the root $z_2^*$. Green: similar points for the root $z_3^*$.Red: similar points for the root $z_4^*$.  Black:  other points. } 
       \label{fig:Degree4}
\end{figure}

\underline{\bf For a polynomial of degree 5:} 

We consider the polynomial $P_5(z)=z^5-3iz^3-(5+2i)z^2+3z+1$, which has 5 roots: $z_1^*\sim -1.28992-1.87357i$, $z_2^*=-0.824853+1.17353i$, $z_3^*=-0.23744+0.0134729i$,  $z_4=0.573868-0.276869i$ and $z_5=1.77834+0.963437i$.  This example is taken from \cite{NewtonFractal}. When applying the standard Newton's method to this polynomial, then the basins of attraction are fractal.  Basins of attraction for Backtracking Gradient Descent seem to be more regular than that for Newton's method, but less regular than that for Backtracking New Q-Newton's method. See Figure \ref{fig:Degree5} for details. 

\begin{figure}
\centering

	 \begin{subfigure}[b]{0.5\textwidth}
        
         \includegraphics[width=\linewidth ]{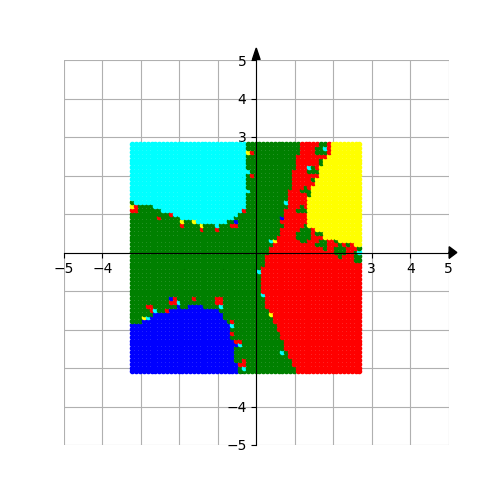}
              \caption{For Newton's method.} 
                      \end{subfigure}%

        \begin{subfigure}[b]{0.5\textwidth}
        
         \includegraphics[width=\linewidth ]{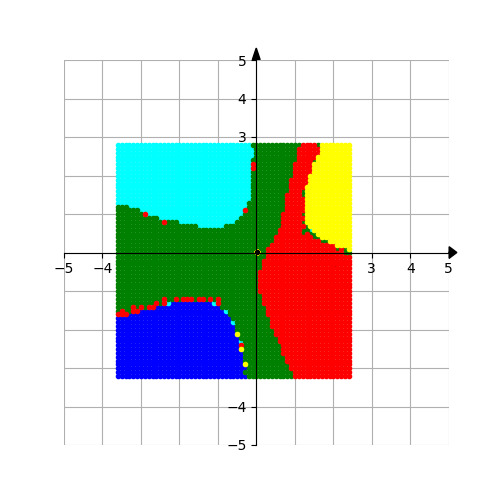}
              \caption{For Backtracking Gradient Descent.} 
                      \end{subfigure}%
        \begin{subfigure}[b]{0.5\textwidth}
              
               \includegraphics[width=\linewidth ]{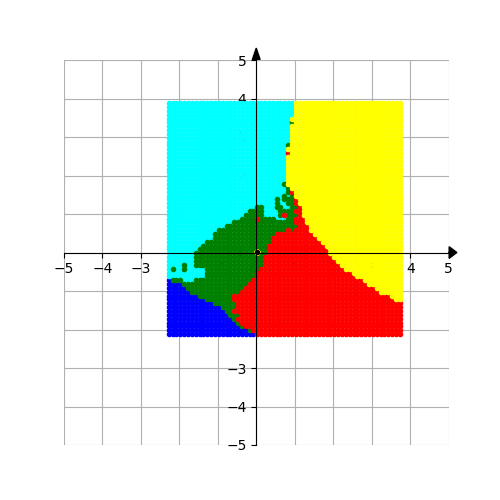}
              \caption{For Backtracking New Q-Newton's method. }

                      \end{subfigure}%
        \caption{Basins of attraction for the polynomial $P_5(z)=z^5-3iz^3-(5+2i)z^2+3z+1$. Blue: initial points $z_0$ for which the constructed sequence converges to $z_1^*$.  Cyan: similar points for the root $z_2^*$. Green: similar points for the root $z_3^*$.  Red: similar points for the root $z_4^*$. Yellow: similar points for the root $z_5^*$. Black:  other points.} 
       \label{fig:Degree5}
\end{figure}

\underline{\bf For a transcendental function:} 

We consider the polynomial $\phi _5(z)=(z^5-3iz^3-(5+2i)z^2+3z+1)e^z$, which has the same  5 roots as the polynomial $P_5(z)$ in the previous example: $z_1^*\sim -1.28992-1.87357i$, $z_2^*=-0.824853+1.17353i$, $z_3^*=-0.23744+0.0134729i$,  $z_4=0.573868-0.276869i$ and $z_5=1.77834+0.963437i$.  This example is taken from \cite{NewtonFractal}. When applying the standard Newton's method to this polynomial, then the basins of attraction are fractal.  Basins of attraction for Backtracking Gradient Descent seem to be more regular than that for Newton's method, but less regular than that for Backtracking New Q-Newton's method. See Figure \ref{fig:Transcendental} for details. 

For this experiment, unlike other experiments presented in this paper, we have to use the version $\widehat{w_k}=w_k/\max \{1,||w_k||\}$ in the Backtracking line search module of Backtracking New Q-Newton's method and of Backtracking Gradient Descent for to avoid overflow in the calculations which one observes when using the simpler version $\widehat{w_k}=w_k$. This can be explained by the fact that the function $\phi _5$ is transcendental and does not have compact sublevels, and indeed can be regarded as having a "zero" at infinity. Backtracking New Q-Newton's method, when using $\widehat{w_k}=w_k$ may  indeed be more easily to create sequences diverging to infinity. On the other hand, Backtracking New Q-Newton's method, using   $\widehat{w_k}=w_k$, may have more regular basins of attraction (where the points diverging to infinity are considered as the basin of the point at infinity). 

\begin{figure}
\centering

	 \begin{subfigure}[b]{0.5\textwidth}
        
         \includegraphics[width=\linewidth ]{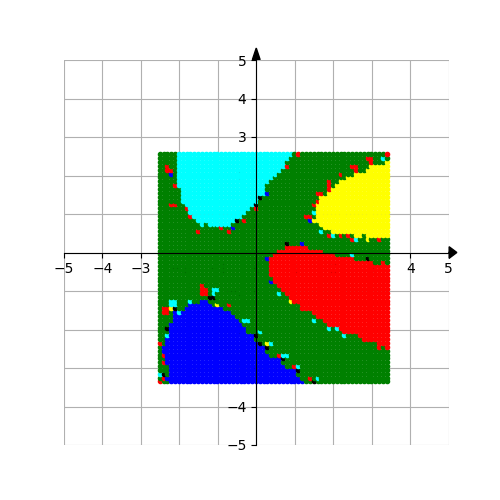}
              \caption{For Newton's method.} 
                      \end{subfigure}%

        \begin{subfigure}[b]{0.5\textwidth}
        
         \includegraphics[width=\linewidth ]{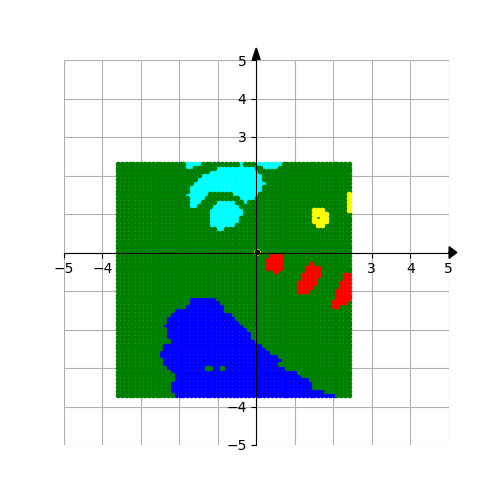}
              \caption{For Backtracking Gradient Descent.} 
                      \end{subfigure}%
        \begin{subfigure}[b]{0.5\textwidth}
              
               \includegraphics[width=\linewidth ]{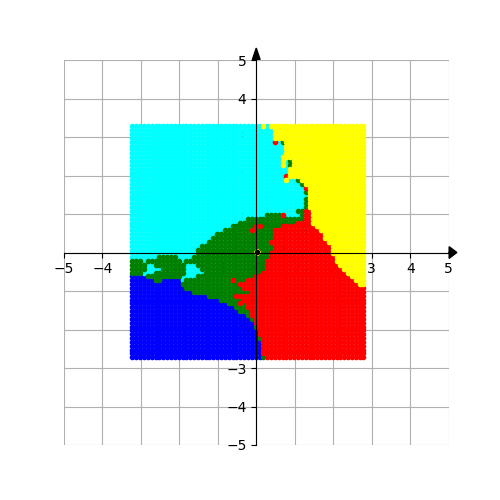}
              \caption{For Backtracking New Q-Newton's method. }

                      \end{subfigure}%
        \caption{Basins of attraction for the transcendental function $\phi _5(z)=(z^5-3iz^3-(5+2i)z^2+3z+1)e^z$. Blue: initial points $z_0$ for which the constructed sequence converges to $z_1^*$.  Cyan: similar points for the root $z_2^*$. Green: similar points for the root $z_3^*$.  Red: similar points for the root $z_4^*$. Yellow: similar points for the root $z_5^*$. Black:  other points.} 
       \label{fig:Transcendental}
\end{figure}

\subsection{Conclusions}

We have considered in this paper the problem of finding minima for a function $f:\mathbb{R}^m\rightarrow \mathbb{R}$, which can treat solving systems of equations in real variables.  We incorporate Backtracking line search into the algorithm New Q-Newton's method defined in \cite{truong-etal} and show that the new resulting algorithm Backtracking New Q-Newton's method has very good theoretical guarantees, on par or better than existing variants of Newton's method. We also propose a generalisation New Q-Newton's method Backtracking G which allows using general symmetric real matrices (and not just the Hessian matrix) and general orthornomal bases of $\mathbb{R}^m$ (and not just eigenvectors of the Hessian matrix). Some of these generalisations have the flavours of quasi-Newton's method. The implementation of the new methods is very direct and easy, and also allows high flexibility.  The algorithms run well on all parameters, even for those not covered by theoretical results currently proven. The experiments illustrate that the new algorithms indeed converge and avoid saddle points, as proven in theoretical results. 

{\bf Future theoretical considerations:} The dynamical aspects of these algorithms present many interesting and new phenomena, and pose challenges to up theory with what observed in experiments. It is worth pointing out that while each of these algorithms are {\bf deterministic} by definition, they are {\bf random} in practice (due to the fact that the search in Backtracking line search has this characteristic), and this helps to enhance their performance. While the results obtained  are already on par or better the more well known algorithms in the current literature), deeper theoretical properties as well as more refined and practical implementations await extensive work in the future. For example, can we rigorously explain the regular features of the basins of attraction of Backtracking New Q-Newton's method (see Figures \ref{fig:Degree2}, \ref{fig:Degree3} and \ref{fig:Degree4}). These figures suggest that this cannot be explained based solely on properties of Backtracking line search (since the basins of attraction for Backtracking Gradient Descent are not that regular), but also properties of the Hessian of the cost function must be taken into account.  

The following question is a first step to understanding more the dynamics of Backtracking New Q-Newton's method. 

{\bf Question 1.} Let $P(z)$ be an analytic function in 1 complex variable. Let's apply Backtracking New Q-Newton's method to the function $f(x,y)=\frac{1}{2}||P(x+iy)||^2$ as usual. a) What are the basins of attraction for zeros of $P(z)$? b) Is it true that the union of these basins of attraction a set of density 1? c)  Do they have fractal structures or are the basins nicely divided? 

Here, we use the following definition for density of a set $A\subset \mathbb{R}^2$. We let for $R>0$ the set $B(R)=\{(x,y)\in \mathbb{R}^2:~x^2+y^2< R\}$. We define $|A|$ to be the Lebesgue measure of $A$. The density of $A$ is defined as $\lim _{R\rightarrow \infty}|A\cap B(R)|/|B(R)|$ (if it exists). 

Note that if $P(z)$ is a generic analytic function, then by results in \cite{truong-etal}, the union of the basins of attractions at the roots of $P(z)$ is a set of full Lebesgue measure, and hence has density 1. Therefore, part b of Question has an affirmative answer for generic analytic functions. For non-generic analytic functions, even for polynomials of degree $\leq 2$, all the 3 parts of Question 1 are totally unknown.  

For the standard Newton's method, a result of Arthur Cayley showed that for a quadratic polynomial $az^2+bz+c$ (with no iterated roots), the picture is very nice: the basin of attraction for each of the two roots is a half plane.  For higher degrees, it is understood from classical results by Fatou and Julia that the basins can have fractal structures \cite{NewtonFractal}, see the next pages for some fractal pictures of Newton's method compared to the more regular shapes of Backtracking New Q-Newton's method. Moreover, it is known from the work \cite{mcmullen} that Newton's method (or more generally iterative rational root-finding algorithms) applied to generic polynomials in general will have exceptional sets of positive Lebesgue measures where initial points chosen in these exceptional sets will not converge to any root of the polynomial. Hence, for these algorithms, it is not meaningful to state Question above.  Figure \ref{fig:Degree2} shows that for degree 2 polynomial,  the basin of Backtracking New Q-Newton's method are very similar to that of the standard Newton's method. In contrast, for polynomials of higher degrees and for transcendental functions,  the basins of attraction of Backtracking New Q-Newton's method seem very regular, and in particular do not seem to have the fractal structure observed in Newton's method.  

{\bf Future experimental/implementational considerations:} Some considerations can help reduce the computational cost for the new algorithms. First, one can use the Two-way Backtracking line search version in \cite{truong-nguyen1}\cite{truong-nguyen2}, which was shown both heuristically and experimentally to reduce the number of iterates and time. Second, one can study simpler versions of Backtracking New Q-Newton's method, for example the version Simplified Backtracking New Q-Newton's method, where just a few eigenvalues and eigenvectors of the Hessian are used. This can be combined with numerical methods, such as Lanczos', to compute the first few eigenvalues and eigenvectors.  

On finding roots of an analytic function in 1 complex variable, one can use two procedures. Let $f:\mathbb{C}\rightarrow \mathbb{C}$ be an analytic function, and $N$ a positive integer. Assume for simplicity that one wants to find  N roots contained in an open bounded connected set  $B\subset \{z\in \mathbb{C}:|z|<R\}$ whose boundary $\partial B$ is a Jordan curve.  In the first procedure, one chooses random initial points $z_0$ inside $B$,  and runs  Backtracking New Q-Newton's method, with the additional criterion that we will stop the algorithm when the constructed sequence starts to leave the circle $\{|z|=R+1\}$ or when the number of roots found reaches $N$. In the second procedure, one chooses random initial points $z_0$ on the boundary $\partial B$,  and runs  Backtracking New Q-Newton's method, with the additional criterion that we will stop the algorithm when the constructed sequence starts to leave the circle $\{|z|=R+1\}$ or when the number of roots found reaches $N$. The second procedure may find N roots quicker than the first procedure. 

We recall a related result in \cite{hubbard-etal}, which for a positive integer $d$, constructed a finite set $A_d\subset \mathbb{C}$ so that for a given polynomial $P_d$ and a root $z^*$ of $P(z)$, Newton's method with at least one initial point  in $A_d$ will converge precisely to $z^*$. This result relies on very deep properties of the roots of a polynomial of a fixed degree and Newton's method, but maybe too rigid (and can be slow, for example when one wants to find only a few roots, say 5, and not all the roots).  Our procedure is to be used for more general functions (and systems of equations). Since a general analytic function can have infinitely many roots, unlike the case of polynomials we need choose in advance the region we want to find the roots and the number of roots to be found. Given a region, one can use Cauchy's integral formulas to obtain a precise (or estimate) value of how many roots it has inside the domain.  

\section{Appendix: How Backtracking New Q-Newton's method finds roots of a real function in 1-dimension $F:\mathbb{R}\rightarrow \mathbb{R}$} In this section we illustrate how Backtracking New Q-Newton's method works in finding roots of a function $F:\mathbb{R}\rightarrow \mathbb{R}$. 

We choose randomly two real numbers $\delta _0, \delta _1$. We fix real  numbers $0<\tau $ and $0<\gamma _0<1$. We define $\kappa =|\delta _0-\delta _1|/2$. We define a non-negative function $f:\mathbb{R}\rightarrow \mathbb{R}$ by the formula: $f(x)=|F(x)|^2/2$. It is easy to check that $f'(x)=F(x)F'(x)$ and $f"(x)=|F'(x)|^2+F(x)F"(x)$. Hence, a critical point of $f$ is exactly a zero of $F$ or a critical point of $F$. Moreover, a non-degenerate critical point $x^*$ of $f$ is a local minimum iff $F(x^*)F"(x^*)>0$, and is a local maximum iff $F(x^*)F"(x^*)<0$. 

Choose randomly an initial point $z_0\in \mathbb{R}$, then construct consecutively a sequence $z_k$ as follows: 

1) If $f'(z_k)=0$ then STOP. If not, go to 2.

2) If $|f"(z_k)+\delta _0|f'(z_k)|^{\tau}|\geq \kappa |f'(z_k)|^{\tau}$ then define $A_k:=f"(z_k)+\delta _0|f'(z_k)|^{\tau}$. Otherwise, define $A_k:=f"(z_k)+\delta _1|f'(z_k)|^{\tau}$. 

3) Define 

\begin{eqnarray*}
w_k&:=&\frac{f'(z_k)}{|A_k|},\\
\widehat{w_k}&:=&\frac{w_k}{\max \{1,|w_k|\}}.
\end{eqnarray*}

Note that 
\begin{eqnarray*}
w_k.f'(z_k)=\frac{|f'(z_k)|^2}{|A_k|}>0.
\end{eqnarray*}

Remark: if $f$ has compact sublevels, then we can choose $\widehat{w_k}=w_k$. 

4) Define $\gamma :=\gamma _0$. While 
\begin{eqnarray*}
f(z_k-\gamma \widehat{w_k})-f(z_k)+\gamma \widehat{w_k}f'(z_k)/(3|A_k|)>0,
\end{eqnarray*}
do $\gamma :=\gamma /3$. 

5) Define $z_{k+1}:=z_k-\gamma \widehat{w_k}$. 

6) Increase $k$ to $k+1$ and repeat the above steps.


\begin{thebibliography}{xx} 

\bibitem{armijo} L. Armijo, {\it Minimization of functions having Lipschitz continuous first partial derivatives}, Pacific J. Math. 16 (1966), no. 1, 1--3. 

\bibitem{acunto-kurdyka} D. D'Acunto and K. Kurdyka, {\it Explicit bounds for the Lojasiewicz exponent in the gradient inequality for polynomials}, Ann. Pol. Math. 87 (2005), 51--61. 

\bibitem{ahookhosh-etal} M. Ahookhosh, F. J. Arag\'on, R. M. T. Fleming and P. T. Vuong, {\it Local convergence of the Levenberg-Marquardt method under H\"older metric subregularity}, Advances in Computational Mathematics, 1--36. 

\bibitem{bianconcini-sciandrone} T. Bianconcini and M. Sciandrone, {\it A cubic regularization algorithm for unconstrained optimization using line search and nonmonotone techniques}, Optimization Methods and Software, 31:5, 1008--1035. 

\bibitem{bolte-etal} J. Bolte, C. Castera, E. Pauwels and C. F\'evotte, {\it An inertial Newton algorithm for Deep Learning}, TSE Working Paper, n.19--1043, October 2019. 

\bibitem{bray-dean} A. J. Bray and and D. S. Dean, {\it Statistics of critical points of gaussian fields on large-dimensional spaces}, Physics Review Letter, 98, 150201. 

\bibitem{cartis-etal} C. Cartis, N. I. M. Gould and P. L. Toint, {\it Adaptive cubic regularisation methods for unconstrained optimization. Part 1: motivation, convergence and numerical results}, Math. Program., Ser. A (2011), 127:245--295. 

\bibitem{dauphin-pascanu-gulcehre-cho-ganguli-bengjo} Y. N. Dauphin, R. Pascanu, C. Gulcehre, K. Cho, S. Ganguli and Y. Bengjo, {\it Identifying and attacking the saddle point problem in high-dimensional non-convex optimization}, NIPS' 14 Proceedings of the 27th International conference on neural information processing systems, Volume 2, pages 2933--2941.  

\bibitem{fr} F. Freudenstein and B. Roth, {\it Numerical solution of systems of nonlinear equations}, Journal of the ACM, volume 10, issue 4, October 1963, pp. 550--556. 

\bibitem{hubbard-etal} J. Hubbard, D. Schleicher and S. Sutherland, {\it How to find all roots of complex polynomials by Newton's method}, Invent. math. 146, 1--33 (2001). 

\bibitem{hueso-etal} J. L. Hueso, E. Mart\'inez and J. R. Torregrosa, {\it Modified Newton's method for systems of nonlinear equations with singular Jacobian}, Journal of Computational and Applied Mathematics, 224 (2009), 77--83. 

\bibitem{levenberg} K. Levenberg, {\it A method for the solution of certain non-linear problems in least squares}, Quarterly of Applied mathematics 2 (2): 164--168.

\bibitem{lojasiewicz2} S. Lojasiewicz, {\it Sur les trajectoires du gradient d'une fonction analytique}, Seminari di Geometria, Bologna 1982/1983, Universita' degli studi di Bologna, Bologna (1984), pp. 115--117.

\bibitem{lojasiewicz1} S. Lojasiewicz, {\it Ensembles semi-analytiques}, preprint IHES, 1965.

\bibitem{mcmullen} C. Macmullen, {\it Families of rational maps and iterative root-finding algorithms}, Ann. of Math. (2), 125 (1987), no 3, 467--493.

\bibitem{marquardt} D. Marquardt, {\it An algorithm for least-squares estimation of nonlinear parameters}, SIAM Journal on Applied Mathematics, 11 (2): 431--441.  

\bibitem{NewtonFractal} Wikipedia page on Newton fractal https://en.wikipedia.org/wiki/Newton$\_$fractal

\bibitem{nesterov-polyak}{Y. Nesterov and B. T. Polyak,} {\it Cubic regularization of Newton method and its global performance}, Math. Program., Ser. A 108, 177--205, 2006.  

\bibitem{panageas-piliouras} I. Panageas and G. Piliouras, {\it Gradient descent only converges to minimizers: Non-isolated critical points and invariant regions}, 8th Innovations in theoretical computer science conference (ITCS 2017), Editor: C. H. Papadimitrou, article no 2, pp. 2:1--2:12, Leibniz international proceedings in informatics (LIPICS), Dagstuhl Publishing. Germany. 

\bibitem{shen-etal} C. Shen, X. Chen and Y. Liang,  {\it A regularized Newton method for degenerate unconstrained optimization problems}, Optimization Letters (2012), 6:1913--1933. 

\bibitem{sumi} H. Sumi, {\it Negativity of Lyapunov exponents and convergence of generic random polynomial dynamical systems and random relaxed Newton's method}, arXiv:1608.05230. 

\bibitem{truong2021} T. T. Truong, {\it New Q-Newton's method meets Backtracking line search: good convergence guarantee, saddle points avoidance, quadratic rate of convergence, and easy implementation}, arXiv:2108.10249.

\bibitem{truongnew} T. T. Truong, {\it Unconstrained optimisation on Riemannian manifolds}, arXiv:2008.11091.

\bibitem{truong4} T. T. Truong, {\it Some convergent results for Backtracking Gradient Descent method on Banach spaces}, arXiv:2001.05768. 

\bibitem{truong} T. T. Truong, {\it Convergence to minima for the continuous version of Backtracking Gradient Descent}, arXiv: 1911.04221. 

\bibitem{truong-nguyen2} T. T. Truong and  H.-T. Nguyen, {\it Backtracking gradient descent method  and some applications to Large scale optimisation. Part 1: Theory}, accepted in Minimax Theory and its Applications. This is the more theoretical part of arXiv: 1808.05160, with some additional experiemental results. 

\bibitem{truong-nguyen1} T. T. Truong and  H.-T. Nguyen, {\it Backtracking gradient descent method and some applications in Large scale optimisation. Part 2: Algorithms and experiments}, published online in Applied Mathematics and Optimization. This is the more applied part of arXiv: 1808.05160, in combination with arXiv:2001.02005 and arXiv:2007.03618. 

\bibitem{truong-etal} T. T. Truong, T. D. To,  H.-T. Nguyen, T. H. Nguyen, H. P. Nguyen and M. Helmy, {\it A fast and simple modification of quasi-Newton's methods helping to avoid saddle points}, arXiv:2006.01512. 

\bibitem{ueda-yamashita2} K. Ueda and N. Yamashita,  {\it A regularized Newton method without line search for unconstrained optimization}, Computational Optimization and Applications (2014), 59:321--351.   

\bibitem{ueda-yamashita1} K. Ueda and N. Yamashita,  {\it Convergence properties of the Regularized Newton method for the unconstrained nonconvex optimization}, Applied Mathematics and Optimization 62, 27--46 (2010).  

\bibitem{num} GitHub link for python's package numdifftools  https://github.com/pbrod/numdifftools

\bibitem{ARCGitHub} GitHub link for Adaptive cubic regularization for Newton's method: https://github.com/cjones6/cubic$\_$reg . Discussion on implementation of this method: https://www.semanticscholar.org/paper/Implementing-the-cubic-and-adaptive-cubic-Jones/5bd648d2840a9753193b39e773ae4192bd43ddd7

\bibitem{tuyenGitHub} GitHub link for Python source codes for Backtracking New Q-Newton's method and variations: https://github.com/tuyenttMathOslo/New-Q-Newton-s-method-Backtracking and https://github.com/tuyenttMathOslo/NewQNewtonMethodBacktrackingForSystemEquations


\end{thebibliography}
\end{document}